\documentclass[11pt,a4paper,twoside]{article}
\usepackage{amsmath,amsfonts,amssymb,amsthm,bbm,latexsym,mathrsfs}
\usepackage{graphicx,color,epsfig,fancyhdr,dsfont}
\usepackage{enumerate}
\usepackage{hyperref}
\usepackage{indentfirst}
\usepackage[all]{hypcap}
\usepackage[affil-it]{authblk}
\allowdisplaybreaks
\title{Numerical Approximation of Random Periodic Solutions of Stochastic Differential Equations}

\author[1]{Chunrong Feng}
\author[1]{Yu Liu}
\author[1]{Huaizhong Zhao}
\affil[1]{Department of Mathematical Sciences, Loughborough
University, LE11 3TU, UK}
\affil[ ]{C.Feng@lboro.ac.uk, Y.Liu4@lboro.ac.uk, H.Zhao@lboro.ac.uk}
\date{}

\newtheorem{theorem}{Theorem}[section]
\newtheorem{definition}[theorem]{Definition}
\newtheorem{lemma}[theorem]{Lemma}
\newtheorem{corollary}[theorem]{Corollary}
\newtheorem{proposition}[theorem]{Proposition}
\newtheorem{remark}[theorem]{Remark}
\newtheorem{example}[theorem]{Example}

\numberwithin{equation}{section}
      
\newcommand{\N}{{\mathbb{N}}}  
\newcommand{\Z}{{\mathbb{Z}}}  
\newcommand{\R}{{\mathbb{R}}}  
\newcommand{\E}{{\mathbb{E}}}  
\newcommand{\X}{{\mathbb{X}}}  

\newcommand{\dt}{{\Delta t}}
\newcommand{\dW}{{\Delta W}}
\newcommand{\dZ}{{\Delta Z}}
\newcommand{\At}{I+A\Delta t}
\newcommand{\at}{1-\alpha\Delta t}

\newcommand{\norm}[1]{\left\lVert#1\right\rVert}
\newcommand{\abs}[1]{\left|#1\right|}
\newcommand{\rbrac}[1]{\left(#1\right)}
\newcommand{\sbrac}[1]{\left[#1\right]}

\newcommand{\hx}{\widehat{X}}
\newcommand{\hxit}[1]{\widehat{X}_{-k\tau+#1\Delta t}^{-k\tau}}
\newcommand{\hu}{\widehat{\Upsilon}}     
\newcommand{\huip}[1]{\widehat{\Upsilon}_+(\widehat{X}_{-k\tau+#1\Delta t}^{-k\tau})}                        
\newcommand{\huim}[1]{\widehat{\Upsilon}_-(\widehat{X}_{-k\tau+#1\Delta t}^{-k\tau})}                        

\newcommand{\hk}[1]{\widehat{K}_{#1}}

\newcommand{\xtr}{X^{-k\tau}_r}
\newcommand{\xts}{X^{-k\tau}_s}

\allowdisplaybreaks

\begin{document}

\maketitle

\begin{abstract}
In this paper, we discuss the numerical approximation of random periodic solutions (r.p.s.) of stochastic differential equations (SDEs) with multiplicative noise. We prove the existence of the random periodic solution as the limit of the pull-back flow when the starting time tends to $-\infty$ along the multiple integrals of the period. As the random periodic solution is not explicitly constructible, it is useful to study the numerical approximation. We discretise the SDE using the Euler-Maruyama scheme and moldiflied Milstein scheme. Subsequently we obtain the existence of the random periodic solution as the limit of the pull-back of the discretised SDE. We prove that the latter is an approximated random periodic solution with an error to the exact one at the rate of $\sqrt {\Delta t}$ in the mean-square sense in Euler-Maruyama method and $\Delta t$ in the Milstein method. We also obtain the weak convergence result for the approximation of the periodic measure.

\noindent
{\bf Keywords:} random periodic solution, periodic measure, Euler-Maruyama method, modified Milstein method, infinite horizon, rate of convergence,  pull-back, weak convergence.
\medskip

\noindent 
{\bf Mathematics Subject Classifications (2000): 37H99, 60H10, 60H35.}
\end{abstract}

\pagestyle{fancy}
\fancyhf{}
\fancyhead[LE,RO]{\thepage}
\fancyhead[LO]{\small{Numerical Approximation of Random Periodic Solutions of SDEs}}
\fancyhead[RE]{\small{C. R. Feng, Y. Liu and  H. Z. Zhao}}

\section {Introduction}

Periodic solution has been a central concept in the theory of dynamical systems since Poincar\'e's pioneering work \cite{poincare}. 
As the random counterpart of periodic solution, the 
concept of random periodic solutions (RPS) began to be addressed recently for a $C^1$-cocycle in \cite{10}. Later the definition of random periodic solutions  and  their existence for semi-flows generated by non-autonomous SDEs and SPDEs with additive noise were given in \cite{11},\cite{12}. 
Denote by $\Delta:=\{(t,s)\in \R^2,s\leq t\}$. Let $\mathbb X$ be a separable Banach space, 
$(\Omega, {\cal F},P, (\theta_t)_{t\in \mathbb R})$ be a metric dynamical system. 
Consider a stochastic periodic semi-flow $u:\Delta\times\Omega\times \X\rightarrow \X$ of period $\tau$, which satisfies the semiflow relation 
\vskip-20pt
\begin{eqnarray}\label{zhao3}
u(t,r,\omega)=u(t,s,\omega)\circ u(s,r,\omega), \label{equality of semi-flow property}
\end{eqnarray}
\vskip-7pt
\noindent
and the periodic property
\vskip-20pt
\begin{eqnarray}\label{zhao4}
u(t+\tau,s+\tau,\omega)=u(t,s,\theta_{\tau}\omega)\label{periodicity of semi-flow property},
\end{eqnarray}
\vskip-4pt
\noindent
for all $r\leq s\leq t$. SDEs and SPDEs with time-dependent coefficients which are periodic in time generate periodic semiflows  
satisfying (\ref{zhao3}) and (\ref{zhao4}) (\cite{11}-\cite {feng-wu-zhao1}).  
\begin{definition} (\cite{11},\cite{12})
A random periodic path of period $\tau$ of the semi-flow $u:\Delta\times\Omega\times \X\rightarrow \X$ is an $\mathcal{F}$-measurable map $Y:\mathbb R\times\Omega\rightarrow \X$ such that for a.e. $\omega\in\Omega$,
$$u(t,s,\omega)Y(s,\omega)=Y(t,\omega),\ Y(s+\tau,\omega)=Y(s,\theta_\tau\omega),\  for\ any\ (t,s)\in\Delta. $$
\end{definition}

It has been proved that random periodic solutions exist for many SDEs and SPDEs (\cite{11}-\cite {feng-wu-zhao1}).
Recently, ``equivalence" of random periodic paths and periodic measures has
been proved in \cite{feng-zhao} and some results of the ergodicity of periodic measures have been obtained.
Note many phenomena in the real world have both periodic and random nature,
e.g. daily temperature, energy consumption, airline passenger volumes, $CO_2$ concentration etc. 
The concept and its study are relevant to modelling random periodicity in the real world.

In literature, there have been a number of recent works such as \cite{chekroun} on random attractors of the stochastic TJ model in climate dynamics; \cite{blw} on stochastic lattice systems; \cite{ch-la-ra-sa} on stochastic resonance; \cite {feng-wu-zhao1} for SDEs with multiplicative linear noise; and \cite{wang} on bifurcations of stochastic reaction diffusion equations. All these results are theoretical on the existence of random periodic paths.

In general, neither stationary solutions nor random periodic solutions can be constructed explicitly, so numerical approximation is another indispensable
tool to study stochastic dynamics, especially to physically relevant problems.
It is worth mentioning here that this is a numerical approximation of an infinite time horizon problem. 
There are numerous work on numerical analysis of SDEs on a finite horizon, and a number of excellent monographs (\cite{5},\cite{milstein}). 
However, there are only a few work on infinite horizon problems. 
A numerical analysis of 
approximation to the stationary solutions and invariant measures of SDEs
through discretising the pull-back,  was given in \cite{mattingly}, \cite{talay}, \cite{talay-tubaro}, \cite{tocino}, \cite{4}. 
Numerical approximations to stable zero solutions of SDEs were given in \cite{hms2003},\cite{5}.

In this paper, we study stochastic differential equations, which possess random periodic solutions and approximate 
them by Euler-Maruyama and Milstein schemes. As far as we know, this is the first paper 
addressing analysis of numerical approximations of random periodic solutions. 
Consider the following m-dimensional SDE 
\begin{eqnarray}\label{feng61}
dX_t^{t_0}=[AX_t^{t_0} + f(t,X_t^{t_0})]dt +g(t,X_t^{t_0})dW_t\label{equations of original period SDE}
\end{eqnarray}
 with $X_{t_0}^{t_0}=\xi $, where $f:\R \times\R^{m}\rightarrow\R^{m},  g:\R \times\R^{m}\rightarrow\R^{m\times d}$, $A$ is a symmetric and negative-definite $m\times m$ matrix, $W_t$ is a two-sided Wiener process in $\R^d$ on a probability space $(\Omega,\mathcal{F},\mathbb{P})$. The filtration is defined as follows 
$\mathcal{F}_s^t=\sigma \{W_u-W_v:s\leq v\leq u\leq t\},\	\mathcal{F}^t=\mathcal{F}_{-\infty}^t=\bigvee_{s\leq t}\mathcal{F}_s^t,$
the random variable $\xi$ is $\mathcal{F}^{t_0}$-measurable. We assume that the functions $f$ and $g$ are $\tau$-periodic in time. 
By the variation of constant formula, the solution of (\ref{equations of original period SDE}) is given
\begin{align}\label{zhao12}
X_t^{t_0}(\xi)=e^{A(t-t_0)}\xi + e^{At}\int_{t_0}^t e^{-As}f(s,X_s^{t_0})ds + e^{At}\int_{t_0}^t e^{-As}g(s,X_s^{t_0})dW_s.
\end{align}	
Denote the standard $P$-preserving ergodic Wiener shift by $\theta:\R\times\Omega\rightarrow\Omega$, 
$\theta_t(\omega)(s):=W(t+s)-W(t),\	t,s\in\R.$
The solution $X$ of the non-autonomous SDE does not satisfy the cocycle property, but $u(t,t_0): \Omega\times R^{m}\to R^m $ given by $u(t,t_0) \xi=X_t^{t_0}(\xi)$ satisfies the semi-flow property (\ref{equality of semi-flow property})
and periodicity (\ref{periodicity of semi-flow property}). 
Denote by $\xtr(\xi,\omega)$ the solution starting from time $-k\tau$.
We will show that when $k\to \infty$, the pull-back $\xtr(\xi)$ has a limit $X^*_r$ in $L^2(\Omega)$ and $X^*_r$ is the random periodic solution of SDE (\ref{equations of original period 
SDE}). It satisfies the infinite horizon stochastic integral equation (IHSIE)
$$X_r^*=\int_{-\infty}^r e^{A(r-s)} f (s, X_s^*)ds+\int _{-\infty}^r e^{A(r-s)}g (s, X_s^*)dW_s.$$
We separate the linear term $AX$ from the nonlinear term in (\ref {feng61}) to enable us to represent the random periodic solution by IHSIE (\cite{11}, \cite{feng-wu-zhao1}).
This is helpful to formulate the scheme for SPDEs for which random periodic solutions were considered in \cite{12}.

Numerical analysis for random periodic solutions was not considered in previous work. 
The infinite horizon stochastic integral equation (IHSIE) method can deal with anticipated cases (\cite {11}-\cite{feng-wu-zhao1}).
But it is still not clear how to numerically approximate two-sided IHSIE and anticipating random periodic solutions.
The pull-back method used in this paper is a popular way to study random attractors. 
Here we use this to deal with stable adapted random periodic solutions of dissipative systems for the first time.  
The pull-back method has some advantages. First, stability can be obtained immediately. Secondly, it can deal with some 
dissipative equations that can not be dealt with by the IHSIE, especially the current IHSIE technique requires equations to have multiplicative linear noise or additive noise and  $f$ being bounded. 
Thirdly in this paper, we study numerical approximations of random periodic solutions
of dissipative SDEs and with the pull-back idea, a random periodic solution of the discretised system can be obtained as well.

 We will first study the Euler-Maruyama numerical scheme in infinite horizon and obtain
an approximating r.p.s. $\widehat X_r^*$. We will prove that 
the latter converges to the exact r.p.s. in $L^2(\Omega)$ at the rate of $\sqrt{\Delta t}$
when the time mesh $\sqrt{\Delta t}$ tends to zero. This result will be numerically verified.
Despite its lower order of the approximation only at the rate of $\sqrt{\Delta t}$, 
the advantage of this scheme is its simplicity and it is relatively easy to implement in actual computations. It works well for the SDE we consider in this paper.

We also consider more advanced numerical schemes, e.g. Milstein scheme (\cite{kloeden1}, \cite{5}, \cite{tocino}), for high order convergence. 
We improve the rate of approximation from $\sqrt{\Delta t}$ in Euler scheme to $\Delta t$.

We will also do some numerical simulations to sample paths of the r.p.s. (Fig. \ref{graph of example 1 samples}). However, simulation of one pathwise trajectory is 
not a reliable way to tell whether or not it is random periodic though it looks very much like to be. Here we provide two reliable methods for this 
from numerical simulations. 
One method is to simulate $\{X^*_t(\omega), t\in {\mathbb R}\}$ and $\{X_t^*(\theta_{-\tau}\omega ), t\in {\mathbb R}\}$ for the same $\omega$. These two trajectories should be repeating each other, but with a shift
of one period of time. See Fig. \ref{graph of example 1 samples} as an example. 
The other way is to simulate  $\{X_t^*(\theta_{-t}\omega), t\in {\mathbb R}\}$, which is periodic if and only if $X_t^*(\omega)$ is random periodic. 
As an example, see Fig. \ref{graph of example 1 periodic}. These two approaches would apply to any other stochastic differential equations
should they have a random periodic solution.

It was known from the recent work \cite {feng-zhao} that the law of the random periodic solution is the periodic measure of the corresponding 
Markov semigroup. Thus we will consider the convergence of
 transition probabilities generated by (\ref{feng61}) and its numerical scheme along the integral multiples of period to the periodic measure and discretised periodic measure respectively and error estimate of the two periodic measures in the weak topology.  

\section{Assumptions and preliminary results}

First we fix some notation. Let $p\geq 1$ and denote the $L^p$-norm of a random variable $\xi$ by $\norm{\xi}_p=\rbrac{\E\abs{\xi}^p}^{1/p}$, and the Frobenius norm of any $d_1\times d_2$ matrix $B$ by $\abs{B}=(\sum_{i=1}^{d_1}\sum_{j=1}^{d_2}B^2_{ij})^{1\over 2}$.

\subsection{Conditions for the SDE}
We assume the following conditions.\\
{\bf Condition (A)}. {\it The eigenvalues of the symmetric matrix $A$, $\lbrace\lambda_j,j=1,2,\ldots,m\rbrace$, satisfy
$0>\lambda_1\geq\lambda_2\geq\ldots\geq\lambda_m.$
}\\
\\
\textbf{Condition (1)}. {\it Assume there exists a constant $\tau>0$ such that for any $t\in \R$, $x\in \R^m$, 
$f(t+\tau,x)=f(t,x),\ 
g(t+\tau,x)=g(t,x).$
and there exist constant $C_0,\beta_1,\beta_2>0$ with $\beta_1+\frac{\beta_2^2}{2}<\abs{\lambda_1}$ such that for any $s,t\in\R$ and $x, y\in\R^m$, 
\begin{eqnarray*}
\abs{f(s,x)-f(t,y)}&\leq C_0\abs{s-t}^{1/2} +\beta_1\abs{x-y},\\
\abs{g(s,x)-g(t,y)}&\leq C_0\abs{s-t}^{1/2} +\beta_2\abs{x-y}.
\end{eqnarray*}
}
\textbf{Condition (2)}. {\it There exists a constant $K^*>0$ such that
$
\norm{\xi}_2\leq K^*.
$
}\\

From Condition (1) it follows that for any $x\in\R^m$, the linear growth condition also holds:
$\abs{f(t,x)}\leq\beta_1\abs{x}+C_1,\ \abs{g(t,x)}\leq\beta_2\abs{x}+C_2,$
where the constants $C_1,C_2>0$ are constants.
It is easy to see that there exists a constant $\alpha$ such that $\beta_1+\frac{\beta_2^2}{2}<\alpha<\abs{\lambda_1}$. In the following, we always assume that $\alpha$ satisfies this condition in all the following proofs. Set $\rho:=\abs{\lambda_m}$. 

 For the SDE case, the quantity $\rho$ is certainly finite and for simplicity, we choose numerical schemes to treat the linear part
 explicitly, which simplify the proof of the pull-back convergence
 to the random periodic solutions for the discretised systems. However, in a case of SPDEs, this technical assumption is no longer true, but can be removed by employing exponential Euler-Maruyama method and Milstein scheme (\cite{baoyuan}, \cite{kloeden2}). This will be studied in future work.

\subsection{Existence and uniqueness of random periodic solution}
%

We first consider the boundedness of the solution in $L^2(\Omega)$.

\begin{lemma}\label{lemma of continuous random periodic solution bounded}
Assume Conditions (A), (1) and (2). Then there exists a constant $C>0$ such that for any $k\in \N$, $r\geq {-k\tau}$, we have $\E\abs{\xtr}^2\leq C.$
\end{lemma}
\begin{proof}
First, using It\^o's formula to $e^{2\alpha r}\abs{X_r^{-k\tau}}^2$, we have
\begin{eqnarray}
e^{2\alpha r}\abs{\xtr}^2=&e^{-2\alpha k \tau}\abs{\xi}^2+2\alpha\int_{-k\tau}^re^{2\alpha s}\abs{\xts}^2ds  +  2\int_{-k\tau}^re^{2\alpha s}\rbrac{\xts}^TA\xts ds\notag\\
&+2\int_{-k\tau}^re^{2\alpha s}\rbrac{\xts}^T  f(s,\xts)ds+\int_{-k\tau}^re^{2\alpha s}\abs{g(s,\xts)}^2ds\notag\\
&+2\int_{-k\tau}^re^{2\alpha s}\rbrac{\xts}^T  g(s,\xts)dW_{s}\label{equation of continuous expansion of ito's formula}
\end{eqnarray}
Firstly note the sum of the second and third terms of the right-hand side is non-positive
as the matrix $\alpha I + A$ is non-positive-definite.
Take the expectation of both sides of (\ref{equation of continuous expansion of ito's formula}), apply the above inequality and use linear growth conditions to obtain
\begin{eqnarray}
 \label{inequality of continuous mild solution in lemma}
e^{2\alpha r}\E\abs{\xtr}^2
&\leq& e^{-2\alpha k \tau}\norm{\xi}_2^2+(2\beta_1+\beta_2^2)\int_{-k\tau}^r e^{2\alpha s}\E\abs{\xts}^2ds\notag\\
&&+2(C_1+\beta_2C_2)\int_{-k\tau}^r e^{2\alpha s}\E\abs{\xts}ds+(2\alpha)^{-1}C_2^2\rbrac{e^{2\alpha r}-e^{-2\alpha k\tau}}.\notag
\end{eqnarray}Also, there exits $\varepsilon>0$, such that
$\rbrac{\beta_1+\frac{\beta_2^2}{2}}(1+\varepsilon )<\alpha<\abs{\lambda_1}.$
By Young's inequality
\begin{align*}
2(C_1+\beta_2C_2)\abs{\xts}\leq\frac{(C_1+\beta_2C_2)^2}{\varepsilon (2\beta_1+\beta_2^2)}+\varepsilon (2\beta_1+\beta_2^2)\abs{\xts}^2.
\end{align*}
Then we have
\begin{align*}
e^{2\alpha r}\E\abs{\xtr}^2\leq &K_1+K_2e^{2\alpha r}+K_3\int_{-k\tau}^r e^{2\alpha s}\norm{\xts}_2^2ds,
\end{align*}
where
\begin{align*}
K_1=&e^{-2\alpha k \tau} \norm{\xi}_2^2-\rbrac{\frac{C_2^2}{2\alpha}+\frac{(C_1+\beta_2C_2)^2}{2\alpha\varepsilon (2\beta_1+\beta_2^2)}}e^{-2\alpha k\tau},\\
K_2=&\frac{C_2^2}{2\alpha}+\frac{(C_1+\beta_2C_2)^2}{2\alpha\varepsilon (2\beta_1+\beta_2^2)},\ 
K_3=(2\beta_1+\beta_2^2)(1+\varepsilon )<2\alpha.
\end{align*}
Now applying Gronwall's inequality, we have
\begin{align*}
e^{2\alpha r}\E\abs{\xtr}^2
\leq& K_1+K_2e^{2\alpha r}+\int_{-k\tau}^r\rbrac{ K_1+K_2e^{2\alpha s}}K_3 e^{\int_s^rK_3dr}ds\\
\leq&(K_1e^{2\alpha k\tau}+K_2)e^{2\alpha r}+\frac{K_2K_3}{2\alpha-K_3}e^{2\alpha r}.
\end{align*}
Here we notice that
$K_1e^{2\alpha k\tau}+K_2=\norm{\xi}_2^2.$
Therefore, by Condition (2)
\begin{align*}
\E\abs{\xtr}^2\leq&  \norm{\xi}_2^2+\frac{2\alpha K_2}{2\alpha-K_3}\leq K^*+\frac{2\alpha K_2}{2\alpha - K_3},
\end{align*}
\vskip-0.9cm \hfill \end{proof}\vskip15pt

In the next lemma, we will also obtain a bound on the norm $\norm{X_{t_1}^{-k\tau}-X_{t_2}^{-k\tau}}_2$ for any fixed time $t_1,t_2$. This will be essential for us to estimate the error of the numerical approximation in Section \ref{zhao2}.
\begin{lemma}\label{lemma of order of dt between different times}
Assume Conditions (A), (1) and (2). Then there exist constants $C_3>0$, $C_4>0$, such that for any positive $k\in \N$ and any $t_1,t_2\geq0,t_1\geq t_2$, the solution of (\ref{equations of original period SDE}) satisfies
$\norm{X_{t_1}^{-k\tau}-X_{t_2}^{-k\tau}}_2\leq C_3(t_1-t_2)+C_4\sqrt{t_1-t_2}.$
\end{lemma}
\begin{proof}
From (\ref{zhao12}), we see that
\begin{eqnarray}
\norm{X_{t_1}^{-k\tau}-X_{t_2}^{-k\tau}}_2
&&\leq e^{2A k \tau}\norm{\xi}_2 \abs{e^{At_1}-e^{At_2}}\notag\\
&&+\norm{e^{At_1}\int_{-k\tau}^{t_1}e^{-As}  f(s,\xts)ds-e^{At_2}\int_{-k\tau}^{t_2}e^{-As}f(s,\xts)ds}_2\notag\\
&&+\norm{e^{At_1}\int_{-k\tau}^{t_1}e^{-As}  g(s,X_{s}^{-k\tau})dW_{s}-e^{A t_2}\int_{-k\tau}^{t_2}e^{-As}g(s,X_{s}^{-k\tau})dW_{s}}_2.\label{inequality of estimation of continuous error between different times}
\end{eqnarray}
We evaluate each term on the right-hand side of (\ref{inequality of estimation of continuous error between different times}). First we consider the first term. By Lemma 1 in \cite {4},
$\abs{e^{At_1}-e^{At_2}}\leq
\abs{A}\rbrac{t_1-t_2}$.
Now we estimate the third term with the Minkowski inequality, It\^o's isometry and the linear growth property 
\begin{align*}
&\norm{e^{At_1}\int_{-k\tau}^{t_1}e^{-As}g(s,X_{s}^{-k\tau})dW_{s}-e^{At_2}\int_{-k\tau}^{t_2}e^{-As}g(s,X_{s}^{-k\tau})dW_{s}}_2\\
\leq&\norm{\int_{-k\tau}^{t_2}\rbrac{e^{At_1}-e^{At_2}} e^{-As}g(s,X_{s}^{-k\tau})dW_{s}}_2+\norm{\int_{t_2}^{t_1}e^{-A(s-t_1)}g(s,X_{s}^{-k\tau})dW_{s}}_2\\
\leq&\sqrt{\int_{-k\tau}^{t_2}\abs{\rbrac{e^{At_1}-e^{At_2}}e^{-As}}^2\E\sbrac{\beta_2\rbrac{\abs{X_{s}^{-k\tau}}}+C_2}^2ds}\\
&+\sqrt{\int_{t_2}^{t_1}\abs{e^{-A(s-t_1)}}^2\E\sbrac{\beta_2\rbrac{\abs{X_{s}^{-k\tau}}}+C_2}^2ds}\\
\leq&\sqrt{\int_{-k\tau}^{t_2}\abs{\rbrac{e^{At_1}-e^{At_2}}e^{-As}}^2\rbrac{2\beta_2^2\E\abs{X_{s}^{-k\tau}}^2+2C_2^2}ds}\\
&+\sqrt{\int_{t_2}^{t_1}\abs{e^{-A(s-t_1)}}^2\rbrac{2\beta_2^2\E\abs{X_{s}^{-k\tau}}^2+2C_2^2}ds}\\
\leq&K_4\sqrt{\int_{-k\tau}^{t_2}\abs{\rbrac{e^{At_1}-e^{At_2}}e^{-As}}^2ds}+K_4\sqrt{\int_{t_2}^{t_1}\abs{e^{-A(s-t_1)}}^2ds}.\\
\leq & K_4\sqrt {{Tr(-A)\over 2}}(t_1-t_2)+K_4 \sqrt {t_1-t_2}.
\end{align*}
Here we take some constant $K_4$ because $\E\abs{\xts}^2$ is bounded above according to Lemma \ref{lemma of continuous random periodic solution bounded}. 
Lastly, we consider the second term of (\ref{inequality of estimation of continuous error between different times}) with Minkowski inequality 
\begin{align*}
&\norm{e^{At_1}\int_{-k\tau}^{t_1}e^{-As}f(s,\xts)ds-e^{At_2}\int_{-k\tau}^{t_2}e^{-As}f(s,\xts)ds}_2\\
\leq&\norm{\int_{-k\tau}^{t_2}(e^{At_1}-e^{At_2})e^{-As}f(s,\xts)ds}_2+\norm{\int_{t_2}^{t_1}e^{-A(s-t_1)}f(s,\xts)ds}_2\\
\leq&\int_{-k\tau}^{t_2}\norm{(e^{At_1}-e^{At_2})e^{-As}f(s,\xts)}_2ds+\int_{t_2}^{t_1}\norm{e^{-A(s-t_1)}f(s,\xts)}_2ds\\
\leq&\int_{-k\tau}^{t_2}\abs{(e^{At_1}-e^{At_2})e^{-As}}\norm{f(s,\xts)}_2ds+\int_{t_2}^{t_1}\abs{e^{-A(s-t_1)}}\norm{f(s,\xts)}_2ds\\
\leq&K_5\rbrac{\int_{-k\tau}^{t_2}\abs{\rbrac{e^{At_1}-e^{At_2}}  e^{-As}}ds+\int_{t_2}^{t_1}\abs{e^{-A(s-t_1)}}ds}\\
\leq& 2K_5(t_1-t_2),
\end{align*}
for a constant $K_5>0$.
Combining the above estimates we obtain the lemma
with the constants $C_3,C_4$ being independent of $k$ and $t_1,t_2$.
\end{proof}

Now we continue to consider the difference of the solutions under various initial values. For simplicity, we here study two different initial values $\xi$ and $\eta$.
\begin{lemma}\label{lemma of continuous random periodic solution wrt different initial values}
Denote by $\xtr$ and $Y_r^{-k\tau}$ two solutions of (\ref{equations of original period SDE}) with different initial values $\xi$ and $\eta$ respectively. Assume Conditions (A), (1) and Condition (2) for both initial values. Then
$\norm{\xtr-Y_r^{-k\tau}}_2
\leq e^{\rbrac{\beta_1+{\beta_2^2\over 2}-\alpha} \rbrac {r+k\tau}} \norm{\xi-\eta}_2.$
\end{lemma}
\begin{proof}
According to (\ref{zhao12}) we have 
\begin{align*}
\xtr-Y_r^{-k\tau}=&e^{A(r+k\tau)}\rbrac{\xi-\eta}+e^{Ar}\int_{-k\tau}^r e^{-As}\rbrac{f(s,\xts)-f(s,Y_s^{-k\tau})}ds\\
&+e^{Ar}\int_{-k\tau}^{r} e^{-As}\rbrac{g(s,X_{s}^{-k\tau})-g(s,Y_{s}^{-k\tau})}dW_s.
\end{align*}
For simplicity, denote $\zeta_r^{-k\tau}=X_r^{-k\tau}-Y_r^{-k\tau}$. Then according to the method used in Lemma \ref{lemma of continuous random periodic solution bounded}, and the global Lipschitz condition, we have
\begin{align*}
e^{2\alpha r}\norm{\zeta_r^{-k\tau}}_2^2
\leq& e^{-2\alpha k \tau} \norm{\xi-\eta}_2^2+2\int_{-k\tau}^r e^{2\alpha s}\E\Big[(\zeta_s^{-k\tau})^T(f(s,\xts)\\
&-f(s,Y_s^{-k\tau}))\Big]ds
+\int_{-k\tau}^r e^{2\alpha s}\E\abs{g(s,\xts)-g(s,Y_s^{-k\tau})}^2ds.\\
\leq& e^{-2\alpha k \tau} \norm{\xi-\eta}_2^2+\rbrac{2\beta_1+\beta_2^2}\int_{-k\tau}^r  e^{2\alpha s}\norm{\zeta_s^{-k\tau}}_2^2ds.
\end{align*}
Then the result follows from the Gronwall inequality.
\end{proof}

Now we can prove the following theorem. 

\begin{theorem}\label{theorem of continuous random periodic solution convergence}
Assume Conditions (A), (1). Then there exists a unique random periodic solution $X^*(r,\cdot)\in L^2(\Omega),r\geq 0$ such that for any initial value $\xi$ satisfying Condition (2), the solution of (\ref{equations of original period SDE}) satisfies
$\lim_{k\rightarrow\infty}\norm{X_r^{-k\tau}(\xi)-X^*(r)}_2=0.$
\end{theorem}
\begin{proof}
Condition (2) implies that the initial value $\xi$ belongs to $L^2(\Omega)$. According to Lemma \ref{lemma of continuous random periodic solution bounded}, $X_r^{-k\tau}(\cdot)$ maps $L^2(\Omega)$ to itself. Now we use the semi-flow property to get that for any $r,k,p\geq 0$,
$X_{r}^{-k\tau-p\tau}(\xi)=X_{r}^{-k\tau}(\omega)\circ X_{-k\tau}^{-(k+p)\tau}(\omega,\xi).$
Thus we can apply Lemma \ref{lemma of continuous random periodic solution wrt different initial values} to have for any $\varepsilon  >0$ there exists $k^*>0$ such that for any $k\geq k^*$,
$\norm{X_r^{-k\tau}(\xi)-X_r^{-(k+p)\tau}(\xi)}_2<\varepsilon .$
This means that there exists $N>0$ such that for any $l,m\geq N$, we have
$\norm{X_r^{-l\tau}(\xi)-X_r^{-m\tau}(\xi)}_2<\varepsilon ,$
i.e.$\{X_r^{-k\tau}(\xi)\}_{k\in\N}$ is a Cauchy sequence,so converges to some $X^*(r,\omega)$ in $L^2(\Omega)$, when $k\rightarrow \infty$.

Set $u(t,r)(\xi)=X_t^r(\xi)$, then $u(t,r):\Omega\times \R^m\rightarrow\R^m$ defines a semi-flow of homeomorphism (Kunita \cite{8}). By the continuity of $X_t^r(\omega):L^2(\Omega, \R^m)\rightarrow L^2(\Omega, \R^m),t\geq r$, then
$u(t,r,\omega)\rbrac{X_r^{-k\tau}(\xi,\omega)} \xrightarrow[L^2(\Omega)]{k\rightarrow\infty} u(t,r,\omega)\circ \rbrac{X^*(r,\omega)}.$
But
$$u(t,r,\omega)\rbrac{X_r^{-k\tau}(\xi,\omega)} = X_t^{-k\tau}(\xi,\omega) \xrightarrow[L^2(\Omega)]{k\rightarrow\infty} X^*(t,\omega).$$
So
$u(t,r,\omega)\rbrac{X^*(r,\omega)} = X^*(t,\omega),\ \mathbb{P} -a.s.$

Taking some other initial value $\eta$ satisfying Condition (2), we have 
\begin{align*}
\norm{X_r^*-X_r^{-k\tau}(\eta)}_2\leq\norm{X_r^*-X_r^{-k\tau}(\xi)}_2+\norm{X_r^{-k\tau}(\xi)-X_r^{-k\tau}(\eta)}_2.
\end{align*}
Applying Lemma \ref{lemma of continuous random periodic solution wrt different initial values} again, we can make the right-hand side small enough when $k\rightarrow \infty$. Therefore the convergence is independent of the initial value.

Now we need to prove the random periodicity of the $X^*(r,\omega)$. Note by the continuity of $f$ and $g$,  
\begin{eqnarray*}
X_{r+\tau}^{-(k-1)\tau}(\xi)
=e^{A(r+k\tau)}\xi+\int_{-k\tau}^{r}  e^{A(r-s)}[f(s,X_{s+\tau}^{-(k-1)\tau}(\xi))ds+g(s,X_{s+\tau}^{-(k-1)\tau}(\xi))d\widetilde W_{s}].
\end{eqnarray*}
where $\widetilde W_s:=(\theta_\tau\omega)(s)=W_{s+\tau}-W_\tau$. On the other hand,
\begin{eqnarray*}
\theta_\tau X_{-k\tau}^{r}(\xi) =e^{A(r+k\tau)}\theta_\tau\xi+\int_{-k\tau}^{r}  e^{A(r-s)}[f(s,\theta_\tau X_{s}^{-k\tau})ds+g(s,\theta_\tau X_{s}^{-k\tau})d\widetilde W_{s}],
\end{eqnarray*}
By pathwise uniqueness of the solution of (\ref{equations of original period SDE}), we have
\begin{eqnarray}
X_r^{-k\tau}(\theta_\tau\omega,\xi(\theta_\tau\omega))=\theta_\tau X_{r}^{-k\tau}(\xi)=X_{r+\tau}^{-(k-1)\tau}(\omega,\xi(\omega)).\label{equality in theorem 2.6}
\end{eqnarray}
From the proof of convergence we have 
\begin{align*}
X_{r+\tau}^{-(k-1)\tau}(\omega,\xi)    \xrightarrow[L^2(\Omega)]{k\rightarrow\infty} X^*(r+\tau,\omega),\ \ 
X_{r}^{-k\tau}(\theta_\tau\omega,\xi(\theta_\tau\omega))  \xrightarrow[L^2(\Omega)]{k\rightarrow\infty} X^*(r,\theta_\tau\omega).
\end{align*}
Therefore $X^*(r+\tau,\omega)=X^*(r,\theta_\tau\omega),\ \mathbb{P}-a.s.$
\end{proof}

\section{Numerical approximation for random periodic solution}
\subsection{Euler-Maruyama scheme}
In this section, we will introduce the basic Euler-Maruyama method to approximate the solution on infinite horizon. Take $\dt=\tau /n$, which will be taken to be sufficiently small such that $\Delta t\leq {1\over \rho}$, for some $n\in \N$,  in the remaining part of the paper. Let $N=kn$. The time domain from time $-k\tau$ to time 0 is divided into $N$ intervals of length $\dt$ such that $N\Delta t=k\tau$. The scheme starts from an $\mathcal{F}^{-k\tau}$-measurable random variable $\xi$ at a time $-k\tau$. At each of the points $i\dt$ we set the value $\hx_{-k\tau+i\dt}^{-k\tau}$ with the iteration formula
\begin{eqnarray}
\hxit{(i+1)}
          &=&\hxit{i}+A\hxit{i}\dt+f(i\dt,\hxit{i})\dt\notag\\
                  & &+g(i\dt,\hxit{i})\rbrac{W_{-k\tau+(i+1)\dt}-W_{-k\tau+i\dt}},\label{{equation of scheme of discrete random periodic solution by steps}}
\end{eqnarray}
where $i=0, 1, 2, \ldots,$ and $\hxit{0}=\xi$.

It is easy to see that for any $M\geq 0$,
\begin{eqnarray}
\hxit{M}&=&(\At)^M\xi+\dt\sum_{i=0}^{M-1}(\At)^{M-i-1}f(i\dt,\hxit{i})\notag\\
&&+\sum_{i=0}^{M-1}(\At)^{M-i-1}g(i\dt,\hxit{i})\rbrac{W_{-k\tau+(i+1)\dt}-W_{-k\tau+i\dt}}\label{equation of scheme of discrete random periodic solution by sum}.
\end{eqnarray}
Moreover, we can set up a discrete semi-flow given by 
$\hat u_{i,j}(\xi)=\hat X_{i\Delta t}^{j\Delta t}(\xi),\  i\geq j,\  i,j\in \{-kn,-kn+1,\cdots \},\ \hat\theta =\theta _{\Delta t}, \ \hat \theta^n=\hat\theta\hat\theta\cdots\hat\theta.$
Then it is easy to see that $u$ satisfies the semi-flow property 
$\hat u_{i,j}(\omega)\circ \hat u_{j,l}(\omega)=\hat u_{i,l}(\omega), \ {\rm for }\ i\geq j\geq l,$
and the periodic property
$\hat u_{i+n,j+n}(\omega)=\hat u_{i,j}(\hat\theta ^n\omega). \ {\rm for }\ i\geq j.$

In order to prove the convergence of the discretized semi-flow to a random periodic solution, we first derive some similar estimates as in Lemma \ref{lemma of continuous random periodic solution bounded} and Lemma \ref{lemma of continuous random periodic solution wrt different initial values}. Then a discrete analogue of Theorem \ref{theorem of continuous random periodic solution convergence} will give us the result.

\begin{lemma}\label{lemma of discrete random periodic solution bounded}
Assume Conditions (A), (1) and (2). Then there exists a constant $\widehat{C}>0$ such that for any natural numbers $k\geq 0$, $M\geq 0$, and sufficiently small $\dt$,  the numerical solution $\hxit{M}$ defined by  (\ref{equation of scheme of discrete random periodic solution by sum}) satisfies
$\E\abs{\hxit{M}}^2\leq\widehat{C}.$ 
\end{lemma} 
\begin{proof}
We still choose $\alpha$ such that
$\beta_1+\frac{\beta^2}{2}<\alpha<\abs{\lambda_1}.$
Then for any $M\geq 0$,
\begin{eqnarray}&&\rbrac{\at}^{-2M}\abs{\hxit{M}}^2\notag\\
&=&\abs{\xi}^2+\sum_{i=0}^{M-1}\rbrac{\at}^{-2i}\rbrac{\frac{\abs{\hxit{(i+1)}}^2}{\rbrac{\at}^2}-\abs{\hxit{i}}^2}.\label{equation of discrete rds expansion}
\end{eqnarray}
This is not hard to verify by expanding the sum and noting cancellations.
Notice that
\begin{eqnarray}
&&\frac{\abs{\hxit{(i+1)}}^2}{\rbrac{\at}^2}-\abs{\hxit{i}}^2\notag\\
&=&\rbrac{\rbrac{\hxit{i}}\rbrac{\frac{\At}{\at}-I}+\frac{\dt}{\at}f(i\dt,\hxit{i})^T \right.\notag\\
&&\hspace{2.5cm}   +\left.\frac{\rbrac{W_{-k\tau+(i+1)\dt}-W_{-k\tau+i\dt}}^T g(i\dt,\hxit{i})^T}{\at}}\notag\\
&&\times\rbrac{\rbrac{\frac{\At}{\at}+I}\hxit{i}+\frac{\dt}{\at}f(i\dt,\hxit{i})\right.\notag\\
&&\hspace{2.5cm}   +\left.\frac{g(i\dt,\hxit{i})\rbrac{W_{-k\tau+(i+1)\dt}-W_{-k\tau+i\dt}}}{\at}}
\label{equation of expansion of square difference}
\end{eqnarray}Note $\rbrac{\frac{\At}{\at}-I}\rbrac{\frac{\At}{\at}+I}$ is non-positive 
definite, where $\dt$ satisfies $0<\dt\leq\frac{1}{\rho}$ as defined before, and for 
each $i$, $f(i\dt,\hxit{i})$ and $g(i\dt,\hxit{i})$ are both independent of $\rbrac{W_{-k\tau+(i
+1)\dt}-W_{-k\tau+i\dt}}$. Take expectation on both sides of (\ref{equation of discrete rds expansion}), consider (\ref{equation of expansion of square difference}), apply the linear growth property and Young's inequality to have
\begin{eqnarray}
&&\hskip 1cm \rbrac{\at}^{-2M}\E\abs{\hxit{M}}^2\label{inequality of discrete approximation in lemma}\\
&\leq&\norm{\xi}_2^2+\sum_{i=0}^{M-1}\rbrac{\at}^{-2i}\rbrac{\frac{\dt}{\at}}^2\E\abs{f(i\dt,\hxit{i})}^2\notag\\
&&+\sum_{i=0}^{M-1}\rbrac{\at}^{-2i}\frac{\dt}{\rbrac{\at}^2}\E\abs{g(i\dt,\hxit{i})}^2\notag\\
&&+\sum_{i=0}^{M-1}\rbrac{\at}^{-2i}\frac{2\dt}{\rbrac{\at}^2}\E\sbrac{\rbrac{\hxit{i}}^T\rbrac{\At}f(i\dt,\hxit{i})}
\notag\\
&\leq&\hk{1}+\rbrac{\at}^{-2M}\hk{2}+\hk{3}\sum_{i=0}^{M-1}\rbrac{\at}^{-2i}\E\abs{\hxit{i}}^2,\notag
\end{eqnarray}where,
\begin{eqnarray*}
&&\hk{1}=\norm{\xi}_2^2,\ \hk{3}=\frac{\dt}{\rbrac{\at}^2}\rbrac{1+\widehat{\varepsilon }}\rbrac{2\beta_1+\beta_2^2+\dt\rbrac{\beta_1^2+2\beta_1\abs{A}}},\\
&&\hk{2}=\frac{C_1^2\rbrac{\dt}^2+C_2^2\dt}{2\alpha\dt-\alpha^2\rbrac{\dt}^2}+\frac{\dt}{2\alpha\dt-\alpha^2\rbrac{\dt}^2}\frac{\rbrac{C_1+\beta_2C_2+\dt C_1\rbrac{\beta_1+\abs{A}}}^2}
{\widehat{\varepsilon }\rbrac{2\beta_1+\beta_2^2+\dt\rbrac{\beta_1^2+2\beta_1\abs{A}}}}.
\end{eqnarray*}
Here $\dt$ and $\widehat{\varepsilon }$ need to be chosen small enough such that
\begin{align*}
\rbrac{1+\widehat{\varepsilon }}\rbrac{2\beta_1+\beta_2^2+\dt\rbrac{\beta_1^2+2\beta_1\abs{A}}}+\alpha^2\dt<2\alpha.
\end{align*}
This guarantees that
$\rbrac{\at}^2\rbrac{1+\hk{3}}<1.$
By the discrete Gronwall inequality,
\begin{eqnarray*}
&&\rbrac{\at}^{-2M}\E\abs{\hxit{M}}^2\\
&\leq&\hk{1}+\hk{2}\rbrac{\at}^{-2M}+\sum_{i=0}^{M-1}\rbrac{\hk{1}+\hk{2}\rbrac{\at}^{-2i}}\hk{3}\rbrac{1+\hk{3}}^{M-i-1}\\
\end{eqnarray*}
It turns out that,
\begin{align*}
\E\abs{\hxit{M}}^2
\leq&\hk{2}+\hk{1}\rbrac{\rbrac{1+\hk{3}}\rbrac{\at}^2}^M\\
&  +\frac{\hk{2}\hk{3}\rbrac{\at}^2\rbrac{1-\rbrac{\rbrac{1+\hk{3}}\rbrac{\at}^2}^M}}{1-\rbrac{1+\hk{3}}\rbrac{\at}^2}
\leq\widehat{C}.
\end{align*}
Note the choice of the constant $\widehat{C}$ is independent of $k$ and the lemma holds for sufficiently small time-step $\dt$ and constant $\widehat{\varepsilon }$. 
\end{proof}

The following lemma is a discrete analogue of Lemma \ref{lemma of continuous random periodic solution wrt different initial values}. 

\begin{lemma}\label{lemma of discrete random periodic solution wrt different initial values}
Denote by $\hxit{M}$ and $\widehat{Y}_{-k\tau+M\dt}^{-k\tau}$ solutions of the Euler scheme with initial values $\xi$ and $\eta$ respectively. Assume Conditions (A), (1) and Condition (2) for both initial values. Let $\dt=\tau/n$, $n\in\Z^+$, be sufficiently small such that $0<\dt\leq \frac{1}{\rho}$. 
Then for any $\varepsilon >0$, there exists an integer $M^*>0$ such that for any $M\geq M^*$, we have
$\norm{\hxit{M}-\widehat{Y}_{-k\tau+M\Delta t}^{-k\tau}}_2<\varepsilon .$
\end{lemma}
\begin{proof}
According to scheme (\ref{equation of scheme of discrete random periodic solution by sum}) we have
\begin{align*}
\hxit{M}-\widehat{Y}_{-k\tau+M\Delta t}^{-k\tau}
=&\rbrac{\At}^M\rbrac{\xi-\eta}+\dt\sum_{i=0}^{M-1}\rbrac{\At}^{M-i-1}\widehat{F}_i\\
&+\sum_{i=0}^{M-1}\rbrac{\At}^{M-i-1}\widehat{G}_i\rbrac{W_{-k\tau+(i+1)\dt}-W_{-k\tau+i\dt}}.
\end{align*}
Here $
\widehat{F}_i=f(i\dt,\hxit{i})-f(i\dt,\widehat{Y}_{-k\tau+i\Delta t}^{-k\tau}),\ 
\widehat{G}_i=g(i\dt,\hxit{i})-g(i\dt,\\\widehat{Y}_{-k\tau+i\Delta t}^{-k\tau}).
$
Denote $\widehat{\zeta}_i=\hxit{i}-\widehat{Y}_{-k\tau+i\Delta t}^{-k\tau}$.
Then by Condition (1), we have $\abs{\widehat{F}_i}\leq\beta_1\abs{\widehat{\zeta}_i}$ and $\abs{\widehat{G}_i}\leq\beta_2\abs{\widehat{\zeta}_i}$.
According to the method used in Lemma \ref{lemma of discrete random periodic solution bounded}, we get the following result similar to inequality (\ref{inequality of discrete approximation in lemma}) 
\begin{align*}
\rbrac{\at}^{-2M}\E\abs{\widehat{\zeta}_M}^2
\leq& \norm{\xi-\eta}_2^2+\sum_{i=0}^{M-1}\rbrac{\at}^{-2i}\rbrac{\frac{\dt}{\at}}^2\E\abs{\widehat{F}_i}^2\\
&+\sum_{i=0}^{M-1}\rbrac{\at}^{-2i}\frac{\dt}{\rbrac{\at}^2}\E\abs{\widehat{G}_i}^2\\
&+\sum_{i=0}^{M-1}\rbrac{\at}^{-2i}\frac{2\dt}{\rbrac{\at}^2}\E\sbrac{\rbrac{\widehat{\zeta}_i}^T\rbrac{\At}\widehat{F}_i}\\
\leq&\norm{\xi-\eta}_2^2+\hk{4}\sum_{i=0}^{M-1}\rbrac{\at}^{-2i}\E\abs{\widehat{\zeta}_i}^2,
\end{align*}
where
$\hk{4}=\frac{\dt}{\rbrac{\at}^2}\rbrac{2\beta_1+\beta_2^2+\dt\rbrac{\beta_1^2+2\beta_1\abs{A}}}.$
We choose $\dt$ small enough such that
$2\beta_1+\beta_2^2+\dt\rbrac{\beta_1^2+2\beta_1\abs{A}}+\alpha^2\dt<2\alpha.$
Then, we have\\
$\rbrac{\at}^2\rbrac{1+\hk{4}}<1.$
Again the discrete Gronwall inequality implies
\begin{align*}
\rbrac{\at}^{-2M}\E\abs{\widehat{\zeta}_M}^2\leq\norm{\xi-\eta}_2^2\prod_{i=0}^{M-1}\rbrac{1+\hk{4}}
=\norm{\xi-\eta}_2^2\rbrac{1+\hk{4}}^M.
\end{align*}
Finally
$\E\abs{\widehat{\zeta}_M}^2\leq\norm{\xi-\eta}_2^2\rbrac{\rbrac{\at}^2\rbrac{1+\hk{4}}}^M<\varepsilon$
with sufficiently large $M$. 
\end{proof}

In the numerical scheme we consider the process as two parts, $[-k\tau,0)$ and $[0,r]$. Define
\begin{eqnarray}\label{eqn3.6}
\hx_r^{-k\tau}:=\hx(r,0,\omega)\circ\hx_0^{-k\tau},
\end{eqnarray}
where $\hx(r,0,\omega)$, $r\geq0$, is finite time Euler approximation of the solution of stochastic differential equation with time step size $\dt$, till $N'\Delta t\leq r$, where $N'$ is the unique number such that $N'\dt\leq r$ and $(N'+1)\dt> r$. If $N'\dt<r$, define
\begin{eqnarray}\widehat{X}(r,0,\omega)=&
\widehat{X}(N'\dt,0,\omega)+f(N'\dt,\widehat{X}(N'\dt,0,\omega))(r-N^{\prime}\dt)\nonumber\\
&+g(N'\dt,\widehat{X}(N'\dt,0,\omega))(W_r-W_{N'\dt})\label{equation of rest part from 0 to r}
\end{eqnarray}

\begin{lemma}(Continuity of the discrete semi-flow with respect to the initial value)\label{lemma of finite time discrete random periodic solution continuity wrt initial value}
Denote by $\widetilde{X}_{r}^{0}$ and $\widetilde{Y}_{r}^{0}$ the solution of the finite time Euler scheme with the initial values $\widetilde{\xi}$ and $\widetilde{\eta}$ at time 0. Assume Conditions (A), (1) and Condition (2) for both initial values. Let $\dt$ be sufficiently small, $p\geq 1$. Then for any $\varepsilon >0$, there exists a $\delta>0$ such that for any $\norm{\widetilde{\xi}-\widetilde{\eta}}_p<\delta$, we have
\begin{eqnarray}
\norm{\widetilde{X}_{r}^{0}(\omega,\widetilde{\xi})-\widetilde{Y}_{r}^{0}(\omega,\widetilde{\eta})}_p<\varepsilon .\label{inequality of lemma continuous wrt iv}
\end{eqnarray}
\end{lemma}
\begin{proof}
Note that $\widetilde{X}_{N'\dt}^0$ and $\widetilde{Y}_{N'\dt}^0$ satisfy analogues of (\ref{equation of scheme of discrete random periodic solution by sum}), with initial value $\widetilde{\xi}$ and $\widetilde{\eta}$ at time 0 instead of $-k\tau$.
Apply the Euler scheme on the finite time $r'=N'\dt$ to obtain 
\begin{eqnarray}
&&\hskip 1cm\abs{\widetilde{X}_{r'}^0(\omega,\widetilde{\xi})-\widetilde{Y}_{r'}^0(\omega,\widetilde{\eta})}^p\label{inequality in continuity lemma expression}\\
&\leq&3^{p-1}\abs{(\At)^{p{N'}}}\abs{\widetilde{\xi}-\widetilde{\eta}}^p+3^{p-1}(\dt)^p\abs{(\At)^{p{N'}}}\abs{\sum_{i=0}^{{N'}-1}(\At)^{-i-1}\widetilde{F}_i}^p\notag\\
&&+3^{p-1}\abs{(\At)^{p{N'}}}\abs{\sum_{i=0}^{{N'}-1}(\At)^{-i-1}\widetilde{G}_i\rbrac{W_{(i+1)\dt}-W_{i\dt}}}^p,\notag
\end{eqnarray}where
$\widetilde{F}_i:=f(i\dt,\widetilde{X}_{i\dt}^0)-f(i\dt,\widetilde{Y}_{i\dt}^0),\ \widetilde{G}_i:=g(i\dt,\widetilde{X}_{i\dt}^0)-g(i\dt,\widetilde{Y}_{i\dt}^0).$
Denote
$\widetilde{\zeta}_i:=\widetilde{X}_{i\dt}^0-\widetilde{Y}_{i\dt}^0.$
For convenience, we denote
$C_p=3^{p-1},\ C_{p,{N'}}=3^{p-1}{N'}^{p-1}.$
Taking expectation on both sides of (\ref{inequality in continuity lemma expression}), and noting that the Lipschitz condition of function $f$ and $g$, we have 
\begin{align*}
(\at)^{-p{N'}}\norm{\widetilde{\zeta}_{N'}}_p^p
\leq&C_p\norm{\widetilde{\xi}-\widetilde{\eta}}_p^p+C_{p,{N'}}(\dt)^p\sum_{i=0}^{{N'}-1}(\at)^{-(i+1)p}\beta_1^p\norm{\widetilde{\zeta}_i}_p^p\\
&+C_{p,{N'}}(\dt)^{p/2}\sum_{i=0}^{{N'}-1}(\at)^{-(i+1)p}\beta_2^p\norm{\widetilde{\zeta}_i}_p^p\\
=&C_p\norm{\widetilde{\xi}-\widetilde{\eta}}_p^p+\widetilde{K}\sum_{i=0}^{{N'}-1}(\at)^{-ip}\norm{\widetilde{\zeta}_i}_p^p,
\end{align*}
where 
$\widetilde{K}=\frac{C_{p,{N'}}\rbrac{(\dt)^p\beta_1^p+(\dt)^{p/2}\beta_2^p}}{(\at)^p},$
which is bounded for any $1\leq p<+\infty$.
Then by the Gronwall inequality, we have
$\norm{\widetilde{\zeta}_{N'}}_p^p\leq C_p\norm{\widetilde{\xi}-\widetilde{\eta}}_p^p\rbrac{(1+\widetilde{K})(\at)^p}^{N'}.$
Note
$(1+\widetilde{K})(\at)^p\leq (\at)^p+C_{p,{N'}}\rbrac{(\dt)^p\beta_1^p+(\dt)^{p/2}\beta_2^p}
\leq 1+C_{p,N'}.$
The result (\ref{inequality of lemma continuous wrt iv}) at $r'=N'\dt$ follows by taking
$\delta=\frac{\varepsilon}{C_p}\rbrac{1+C_{p,N'}}^{-{N'}}.$
Finally (\ref{inequality of lemma continuous wrt iv}) at time $r$ follows from (\ref{equation of rest part from 0 to r}) and the estimate at $r'=N'\dt$.
\end{proof}

\begin{theorem}\label{theorem of discrete random periodic solution convergence}
Assume that Condition (1)  and $\dt$ is fixed and small enough. The time domain is divided as $\tau=n\dt$. Then there exists $\hx_r^*\in L^2\rbrac{\Omega}$ such that for any initial values $\xi$ satisfying Condition (2), the solution of the Euler-Maruyama scheme satisfies
\begin{eqnarray}\lim_{k\rightarrow\infty}\norm{\hx_{r}^{-k\tau}\rbrac{\xi}-\hx_r^*}_2=0,\label{equation of the result of theorem discrete convergence}
\end{eqnarray}and $\hx_r^*$ satisfies the random periodicity property.
\end{theorem}
\begin{proof}
Firstly we note that the proof of the convergence of the process $\hx_0^{-k\tau}$ can be made similarly as that of Theorem \ref{theorem of continuous random periodic solution convergence}. According to Lemma \ref{lemma of discrete random periodic solution bounded} we know that for any $M$, we have $\hxit{M}\in L^2\rbrac{\Omega}$. We use a similar construction of a Cauchy sequence as in Theorem \ref{theorem of continuous random periodic solution convergence}. As we assume that $\tau=n\dt$ and $k\tau=kn\dt=:N\dt$, we have the following result by using semi-flow property, for any $m\geq 1$,
\begin{align*}
\hx_0^{-(k+m)\tau}=\hx_{0}^{-(N+mn)\dt}=\hx_{0}^{-N\dt} \circ \hx_{-N\dt}^{-(N+mn)\dt}.
\end{align*}
It is a same process as $\hx_0^{-N\Delta t}$ with a different initial value. By Lemma \ref{lemma of discrete random periodic solution wrt different initial values} we have that for any $\varepsilon >0$ there exists $N^*$ such that for any $N\geq N^*,\dt>0$, we have 
\begin{align*}
\norm{\hx_0^{-k\tau}-\hx_0^{-(k+m)\tau}}_2=\norm{\hx_0^{-N\dt}-\hx_0^{-(N+mn)\dt}}_2<\varepsilon .
\end{align*}	
Then we construct the Cauchy sequence $\hx_i=\hx_0^{-i\tau}$, which converges to some $\hx^*$ in $L^2\rbrac{\Omega}$. We now use the same method to prove the convergence is independent of the initial point. Note for fixed $\dt$,
\begin{align*}
\norm{\hx^*-\hx_{0}^{-k\tau}\rbrac{\eta}}_2\leq\norm{\hx^*-\hx_{0}^{-k\tau}(\xi)}_2+\norm{\hx_{0}^{-k\tau}(\xi)-\hx_{0}^{-k\tau}\rbrac{\eta}}_2\xrightarrow{N\rightarrow\infty} 0,
\end{align*}
where $N \rightarrow \infty$ is equivalent to $k \rightarrow \infty$.

Define $\hx^*(r,\omega):= \hx(r,0,\omega)\circ \hx^*$, $r\geq 0$. According to Lemma \ref{lemma of finite time discrete random periodic solution continuity wrt initial value}, we have
$$\hx_{r}^{-k\tau}(\omega)=\hx(r,0,\omega)\circ \hx_{0}^{-k\tau}(\omega)\xrightarrow[L^2(\Omega)]{k\rightarrow\infty}\hx(r,0,\omega)\circ \hx^*(\omega)=\hx^*(r,\omega),$$
so (\ref{equation of the result of theorem discrete convergence}) holds. On the other hand, similar to the proof of (\ref{equality in theorem 2.6}), we obtain
$$\hx_{r+\tau}^{\tau}(\omega,\xi(\omega))=\hx_{r}^{0}(\theta_\tau\omega,\xi(\theta_\tau\omega))=\theta_\tau\hx_r^0(\omega,\xi(\omega)).$$
Therefore,
\begin{align*}
\hx_{r}^{-k\tau}(\theta_\tau\omega)=\hx(r,0,\theta_\tau\omega)\circ\hx_0^{-k\tau}(\theta_\tau\omega) &\xrightarrow[L^2(\Omega)]{k\rightarrow\infty}\hx(r,0,\theta_\tau\omega)\circ\hx^*(\theta_\tau\omega)=\hx^*(r,\theta_\tau\omega).
\end{align*}
But,
$$\hx_{r+\tau}^{-k\tau+\tau}(\omega) \xrightarrow[L^2(\Omega)]{k\rightarrow\infty} \hx^*(r+\tau,\omega),\ {\rm and}\ 
\hx_{r+\tau}^{-k\tau+\tau}(\omega)=\hx_{r}^{-k\tau}(\theta_\tau\omega),\mathbb{P}-a.s,$$
thus we have
$\hx^*(r+\tau,\omega)=\hx^*(r,\theta_\tau\omega),\mathbb{P}-a.s.$
\end{proof}
\begin{example}
Consider
a specific SDE
\begin{eqnarray}\label{zhao11}
dX_t^{t_0}=-\pi X_t^{t_0}dt  + \sin(\pi t)dt + X_t^{t_0}dW_t.\label{equation of example 1}
\end{eqnarray} 
According to Theorem \ref{theorem of continuous random periodic solution convergence}, (\ref{zhao11}) has a random periodic solution. 
 By Theorem
\ref{theorem of discrete random periodic solution convergence}, its Euler-Maruyama dissertation also has a random periodic path.
To see the ``periodicity" numerically, we provided two methods.  
One approach is to simulate the processes $\hat X_t^*(\omega)=\hx_t^{-6}(\omega, 0.5), -5\leq t\leq 0$, and $\hat X_t^*(\theta _{-2}\omega)=\hx_t^{-6}(\theta_{-2}\omega,0.5), -5\leq t\leq 2$, with the same $\omega$ 
and step size $\Delta t=0.01$ (Fig. \ref{graph of example 1 samples}). One can see that these two trajectories exactly repeat each with a time shift of one period (only comparing the graph of $\hat X_t^*(\theta _{-2}\omega)$ for $-3\leq t\leq 2$).
The second method is the simulation of  $\{\hat X_t^*(\theta_{-t}\omega), 0\leq t\leq 6\}$ for the same realisation $\omega$ and step size as before (Fig. \ref{graph of example 1 periodic}). One can easily see that Fig. \ref{graph of example 1 periodic} is a perfect periodic curve. This agrees with the fact that if $\hat X_t^*(\omega)$ is a random periodic path  iff $\hat X_t^*(\theta_{-t}\omega)$ is periodic, i.e. $\hat X^*_{t+\tau}(\theta_{-(t+\tau)}\omega)=\hat X_t^*(\theta_{-t}\omega)$. Note in theory $ \hat X_t^*=\hat X_t^{-\infty}$, but we take pull-back time $-6$ as this is already enough to generate a good convergence to the random periodic paths $\hat X_t^*(\cdot)$  for $t\geq -5$ by the solution starting at $-6$ from $0.5$  
for both cases.  The choice of the initial position does not affect random periodic paths, but  the time to take for the convergence. 
\vskip-5pt
\begin{figure}
\begin{minipage}{\textwidth}
\centering 
\includegraphics[scale=0.6]{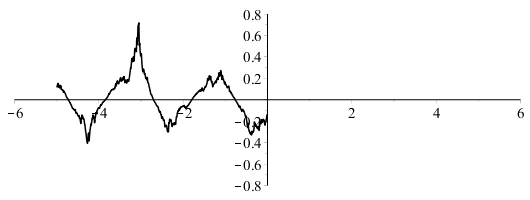}\\
\includegraphics[scale=0.6]{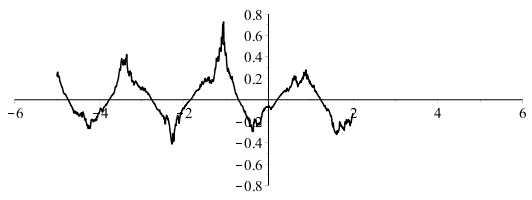}
\end{minipage} 
\caption{Simulations of the processes $\{\hat X_t^*(\omega), -5\leq t\leq 0\}$ and $\{\hat X_t^*(\theta _{-2}\omega), -5\leq t\leq 2\}.$}\label{graph of example 1 samples}
\end{figure}
\begin{figure}
\centering 
\includegraphics[scale=0.6]{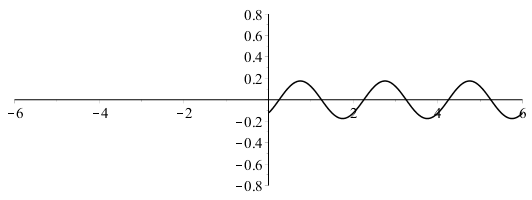}
\caption{Simulation of the process
$\{\hat X_t^*(\theta_{-t}\omega), 0\leq t\leq 6\}$}.
\label{graph of example 1 periodic}
\end{figure}
\end{example}
\subsection{Modified Milstein scheme}
We will consider the Milstein scheme which will increase the convergence order for the infinite horizon problem.\\
\\
\textbf{Condition (1$^{\prime}$)}. {\it Assume there exists a constant $\tau>0$ such that for any $t\in \R$, $x\in \R^m$, 
$f(t+\tau,x)=f(t,x),\ 
g(t+\tau,x)=g(t,x),$
and there exist constants $C_0,\beta_1,\beta_2>0$ with $\beta_1+\frac{\beta_2^2}{2}<\abs{\lambda_1}$ such that for any $s,t\in\R$ and $x\in\R^m$, 
\begin{eqnarray*}
\abs{f(s,x)-f(t,y)}&\leq C_0\abs{s-t} + \beta_1\abs{x-y},\\
\abs{g(s,x)-g(t,y)}&\leq C_0\abs{s-t} + \beta_2\abs{x-y}.
\end{eqnarray*}
Meanwhile, we assume the boundedness of first order partial derivative of function $f$ and $g$ with respect to $x$. 
}
\\

The iteration formula for the modified SRK scheme is
\begin{eqnarray}
\hxit{(i+1)}
          &=&\hxit{i}+A\hxit{i}\dt+f(i\dt,\hxit{i})\dt\notag\\
          &&+g(i\dt,\hxit{i})\rbrac{\dW_i}\label{equation of scheme of discrete random periodic solution by steps for RK1}\\
          &&+\frac{\dZ_i}{2\sqrt{\dt}}\sbrac{f\rbrac{i\dt,\huip{i})}-f\rbrac{i\dt,\huim{i}} }    \notag\\
          &&\hspace{-0.3cm}+\frac{(\dW_i)^2-\dt}{4\sqrt{\dt}}\sbrac{g\rbrac{i\dt,\huip{i})}-g\rbrac{i\dt,\huim{i}} }    \notag,
\end{eqnarray}
with
\begin{align*}
\hu_\pm(\hxit{i})=&\hxit{i}+A\hxit{i}\dt+f(i\dt,\hxit{i})\dt\\
&\pm g(i\dt,\hxit{i})\sqrt{\dt}
\end{align*}
and
\begin{align*}
\dW_i&=\int_{-k\tau+i\dt}^{-k\tau+(i+1)\dt}dW_s= W_{-k\tau+(i+1)\dt}-W_{-k\tau+i\dt},\\
\dZ_i&=\int_{-k\tau+i\dt}^{-k\tau+(i+1)\dt}\int_{-k\tau+i\dt}^s dW_u ds,
\end{align*}
where $i=0, 1, 2, \ldots,$ and $\hxit{0}=\xi$. Here we used the approximation of $\dZ_i$ by the method of Kloeden and Platen in \cite{5}.

\begin{theorem}\label{thm10}
Assume that Conditions (A), ($1^{\prime}$)  hold and $\dt$ is fixed and small enough. The time domain is divided as $\tau=n\dt$. Then there exists $\hx_r^*\in L^2\rbrac{\Omega}$ such that for any initial values $\xi$ satisfying Condition (2), the solution of the Milstein scheme satisfies
\begin{eqnarray}\lim_{k\rightarrow\infty}\norm{\hx_{r}^{-k\tau}\rbrac{\xi}-\hx_r^*}_2=0,
\end{eqnarray}and $\hx_r^*$ satisfies the random periodicity property.
\end{theorem}
\begin{proof}
The proof is by a similar argument as Theorem \ref{theorem of discrete random periodic solution convergence}. As it is tedious and there is no special difficulty, so omitted here.
\end{proof}
\begin{remark}
For the Milstein scheme, the existence of constants as $\hk{1},\hk{2},\hk{3},\hk{4}$ in the proof of Euler-Maruyama scheme are guaranteed by the boundedness of partial derivatives of functions $f$ and $g$. Then we still have the convergence for different initial values and the boundedness of the discrete processes. The addition term $\frac{\dZ_i}{2\sqrt{\dt}}[f\rbrac{i\dt,\huip{i})}-f\rbrac{i\dt,\huim{i}} ]$ in the scheme does not influence the result of the convergence. However, when we analyse the error between approximation and the exact solution of random periodic solutions, this term is necessary for infinite horizon case to satisfy the order of error.
\end{remark}

\section{The error estimate}
\label{zhao2}
\subsection{Euler-Maruyama method}
In the last two sections, we proved the existence of random periodic solutions of SDE (\ref{equations of original period SDE}) and its discretisations as the limits of semi-flows when the starting times were pushed to $-\infty$. The next step is to estimate the error between these two limits. Now we need to consider the difference between the discrete approximate solution and the exact solution. The exact solution at time $-k\tau+M\dt$ is as follows 
\begin{eqnarray}X_{-k\tau+M\Delta t}^{-k\tau}(\omega,\xi)=&e^{AM\dt}\xi+e^{A(M\dt-k\tau)}\int_{-k\tau}^{M\dt-k\tau}e^{-As}f(s,X_s^{-k\tau})ds\notag\\
&+e^{A(M\dt-k\tau)}\int_{-k\tau}^{M\dt-k\tau}e^{-As}g(s,X_s^{-k\tau})dW_{s}\label{equation of continuous random periodic solution into N dt}.
\end{eqnarray}
\begin{lemma}\label{lemma of continuous and discrete random periodic solution error estimation from -kt to 0}
Assume Conditions (A), (1) and (2). Choose $\dt=\tau/n$ for some $n\in \N$ and $N=kn$. Then there exists a constant $K>0$ such that for any sufficiently small fixed $\dt$ and $N'\in {\mathbb N}$, , we have 
$$\limsup_{k\rightarrow\infty}\norm{X_{N'\Delta t}^{-k\tau}-\hx_{N'\Delta t}^{-k\tau}}_2\leq K\sqrt{\dt},$$
where $X_{N'\Delta t}^{-k\tau}$ and $\hx_{N'\Delta t}^{-k\tau}$ are the exact and the numerical solutions given by (\ref{equation of continuous random periodic solution into N dt}) and (\ref{equation of scheme of discrete random periodic solution by sum}) respectively, $K$ is independent of $N'$ and $\Delta t$.
\end{lemma}
\begin{proof}
In the following proof, we always denote by $\hat K_\cdot$ the constant derived from the unlderlining computation unless otherwise stated.
For any $M\in {\mathbb N}$, we have
\begin{align*}
&X_{-k\tau+M\Delta t}^{-k\tau}-\hxit{M}\\
=&\rbrac{e^{AM\dt}-\rbrac{\At}^M}\xi+e^{A(M\dt-k\tau)}\int_{-k\tau}^{M\dt-k\tau}e^{-As}f(s,X_s^{-k\tau})ds\\
&-\sum_{i=0}^{M-1}\rbrac{\At}^{M-i-1}f(i\dt,\hxit{i})\dt\\
&
+e^{A(M\dt-k\tau)}\int_{-k\tau}^{M\dt-k\tau}e^{-As}g(s,X_s^{-k\tau})dW_{s}\\
& -\sum_{i=0}^{M-1}\rbrac{\At}^{M-i-1}g(i\dt,\hxit{i})\rbrac{W_{-k\tau+(i+1)\dt}-W_{-k\tau+i\dt}}.
\end{align*}
Similar to the method of Lemma \ref{lemma of discrete random periodic solution bounded}, firstly consider
\begin{eqnarray}
&&\hspace{0.7cm}\rbrac{\at}^{-2M}\abs{X_{-k\tau+M\Delta t}^{-k\tau}-\hxit{M}}^2\label{equation of discrete rds expansion in chapter 4}\\
&=&\sum_{i=0}^{M-1}\rbrac{\at}^{-2i}\rbrac{\frac{\abs{X_{-k\tau+(i+1)\Delta t}^{-k\tau}-\hxit{(i+1)}}^2}{\rbrac{\at}^2}-\abs{X_{-k\tau+i\Delta t}^{-k\tau}-\hxit{i}}^2}.\notag
\end{eqnarray}
For simplicity we denote
\begin{align*}
B_1=&\frac{1}{\at}\int_{i\dt-k\tau}^{(i+1)\dt-k\tau}\rbrac{e^{-A\rbrac{s+k\tau-(i+1)\dt}}f(s,X_s^{-k\tau})-f(i\dt,\hxit{i})}ds,\\
B_2=&\frac{1}{\at}\int_{i\dt-k\tau}^{(i+1)\dt-k\tau}\rbrac{e^{-A\rbrac{s+k\tau-(i+1)\dt}}g(s,X_{s}^{-k\tau})-g(i\dt,\hxit{i})}dW_s.
\end{align*}
Therefore,
\begin{eqnarray*}
&&X_{-k\tau+(i+1)\Delta t}^{-k\tau}-\hxit{(i+1)}\\
&=&e^{A\dt}X_{-k\tau+i\Delta t}^{-k\tau}-\rbrac{\At}\hxit{i}+\rbrac{\at}\rbrac{B_1+B_2}.
\end{eqnarray*}
Now we consider
\begin{eqnarray}
&&\hskip 1cm\frac{\abs{X_{-k\tau+(i+1)\Delta t}^{-k\tau}-\hxit{(i+1)}}^2}{\rbrac{\at}^2}-\abs{X_{-k\tau+i\Delta t}^{-k\tau}-\hxit{i}}^2\label{equation of expansion of sum part}\\
&=&\rbrac{X_{-k\tau+i\Delta t}^{-k\tau}-\hxit{i}}^T  \rbrac{\frac{e^{A\dt}}{\at}-I}  \rbrac{\frac{e^{A\dt}}{\at}+I} \notag\\
&& \hskip3cm \times  \rbrac{X_{-k\tau+i\Delta t}^{-k\tau}-\hxit{i}}\notag\\
&&+\rbrac{\hxit{i}}^T  \rbrac{\frac{e^{A\dt}-I-A\dt}{\at}}^2  \rbrac{\hxit{i}}  +B_1^TB_1+B_2^TB_2\notag\\
&&+2\rbrac{X_{-k\tau+i\Delta t}^{-k\tau}-\hxit{i}}^T  \rbrac{\frac{e^{A\dt}}{\at}}  \rbrac{\frac{e^{A\dt}-I-A\dt}{\at}}   \rbrac{\hxit{i}}\notag\\
&&+2\rbrac{\rbrac{X_{-k\tau+i\Delta t}^{-k\tau}}^T  \rbrac{\frac{e^{A\dt}}{\at}}  -  \rbrac{\hxit{i}}^T  \rbrac{\frac{\At}{\at}}}  B_1\notag\\
&&+2\rbrac{\rbrac{X_{-k\tau+i\Delta t}^{-k\tau}}^T  \rbrac{\frac{e^{A\dt}}{\at}}  -  \rbrac{\hxit{i}}^T  \rbrac{\frac{\At}{\at}}}  B_2  +2B_1^TB_2\notag.
\end{eqnarray}We note that the matrix $\rbrac{\frac{e^{A\dt}}{\at}-I}  \rbrac{\frac{e^{A\dt}}{\at}+I}$ can be non-positive-definite when we choose the $\dt$ small enough. Now we consider each term in (\ref{equation of expansion of sum part}). First,
\begin{align*}
&\E\sbrac{\rbrac{\hxit{i}}^T  \rbrac{\frac{e^{A\dt}-I-A\dt}{\at}}^2  \hxit{i}}\\
\leq&\norm{\hxit{i}}_2  \abs{\frac{\frac{1}{2}A^2\rbrac{\dt}^2}{\at}}^2  \norm{\hxit{i}}_2
\leq \hk{5}(\dt)^4.
\end{align*}
Next,
\begin{eqnarray}
&&\hskip 1cm\E\sbrac{B_1^TB_1}=\E\abs{B_1}^2\label{equation of random periodic solution f part}\\
&\leq&  \frac{2(1+\mu)}{\mu\rbrac{\at}^2}\rbrac{\int_{i\dt-k\tau}^{\rbrac{i+1}\dt-k\tau}  \abs{e^{-A\rbrac{s+k\tau-(i+1)\dt}}-I}  \norm{f(s,X_s^{-k\tau})  }_2  ds}^2\notag\\
&&+  \frac{2(1+\mu)}{\mu\rbrac{\at}^2}\rbrac{\int_{i\dt-k\tau}^{\rbrac{i+1}\dt-k\tau}  \norm{f(s,X_s^{-k\tau})  -f(i\dt,X_{-k\tau+i\Delta t}^{-k\tau})  }_2  ds}^2\notag\\
&&+  \frac{1+\mu}{\rbrac{\at}^2}\rbrac{\int_{i\dt-k\tau}^{\rbrac{i+1}\dt-k\tau}  \norm{f(i\dt,X_{-k\tau+i\Delta t}^{-k\tau})  -f(i\dt,\hxit{i})  }_2  ds}^2\notag,
\end{eqnarray}where $\mu$ is a small number from Young's inequality, which will be fixed later.
By linear growth property of $f$ and Lemma \ref{lemma of continuous random periodic solution bounded}, we know that $\norm{f(s,X_s^{-k\tau})  }_2$ is bounded. So for the first term in (\ref{equation of random periodic solution f part}) we only need to estimate
\begin{align*}
&\int_{i\dt-k\tau}^{(i+1)\dt-k\tau}  \abs{e^{-A\rbrac{s+k\tau-(i+1)\dt}}-I}  ds
\leq\frac{(\dt)^2}{2} Tr\rbrac{  -A  }.
\end{align*}
By Condition (1) and Lemma \ref{lemma of order of dt between different times}, the second term in (\ref{equation of random periodic solution f part}) becomes
\begin{align*}
&\int_{i\dt-k\tau}^{(i+1)\dt-k\tau}  \norm{  f(s,X_s^{-k\tau}) - f(i\dt,X_{-k\tau+i\Delta t}^{-k\tau})}_2  ds\\
\leq&\int_{i\dt-k\tau}^{(i+1)\dt-k\tau}  (\norm{f(s,X_s^{-k\tau}) - f(i\dt,X_s^{-k\tau})}_2 
\\
&
\ \ \ \ \ \ \ \ \ \ \ \ \ \ \ \ \ \ \ \ \ \ \ +\norm{f(i\dt,X_s^{-k\tau}) - f(i\dt,X_{-k\tau+i\Delta t}^{-k\tau})}_2 )ds\\
\leq&\int_{i\dt-k\tau}^{(i+1)\dt-k\tau} C_0\abs{s-i\dt+k\tau}^{1/2} ds + \int_{i\dt-k\tau}^{(i+1)\dt-k\tau} \beta_1 \norm{X_s^{-k\tau}-X_{-k\tau+i\Delta t}^{-k\tau}}_2 ds\\
\leq&\int_{i\dt-k\tau}^{(i+1)\dt-k\tau} (C_0+\beta_1 C_4) \sqrt{s-i\dt+k\tau}ds\\
\leq & 
\hk{6} \rbrac{\dt}^\frac{3}{2}.
\end{align*}
Applying the global Lipschitz condition, the third term of (\ref{equation of random periodic solution f part}) becomes 
\begin{align*}
&\int_{i\dt-k\tau}^{(i+1)\dt-k\tau}  \norm{f(i\dt,X_{-k\tau+i\Delta t}^{-k\tau})  -f(i\dt,\hxit{i})}_2  ds\\
\leq & \beta_1\dt\norm{X_{-k\tau+i\Delta t}^{-k\tau}-\hxit{i}}_2.
\end{align*}
We summarise the above inequalities to have
\begin{align}\label{eqn4.5}
\E\sbrac{B_1^TB_1}\leq \hk{7}\rbrac{\dt}^3  +\frac{(1+\mu)\beta_1^2\rbrac{\dt}^2}{\rbrac{\at}^2}    \norm{X_{-k\tau+i\Delta t}^{-k\tau}-\hxit{i}}_2^2 .
\end{align}
This term is of the 3rd order of $\dt$ and 2nd order of $\dt$ with $ \norm{X_{-k\tau+i\Delta t}^{-k\tau}-\hxit{i}}_2^2$.\\
Similar to the $\E\sbrac{B_1^TB_1}$, the following term can be estimated as
\begin{eqnarray}
&&\E\sbrac{B_2^TB_2}=\E\abs{B_2}^2\notag\\
&\leq&  \frac{2(1+\mu)}{\mu\rbrac{\at}^2}  \int_{i\dt-k\tau}^{\rbrac{i+1}\dt-k\tau}     \abs{ e^{-A\rbrac{s+k\tau-(i+1)\dt}} -I  }^2 \norm{  g(s,X_s^{-k\tau}) }^2_2     ds \notag\\
&&+  \frac{2(1+\mu)}{\mu\rbrac{\at}^2}  \int_{i\dt-k\tau}^{\rbrac{i+1}\dt-k\tau}  \norm{  g(s,X_{s}^{-k\tau}) -g(i\dt,X_{-k\tau+i\Delta t}^{-k\tau})  }^2_2  ds\notag\\
&&+  \frac{1+\mu}{\rbrac{\at}^2}  \int_{i\dt-k\tau}^{\rbrac{i+1}\dt-k\tau}  \norm{  g(i\dt,X_{-k\tau+i\Delta t}^{-k\tau})  -g(i\dt,\hxit{i})   }^2_2  ds \label{equation of random periodic solution g part},
\end{eqnarray}
where $\mu$ is a small number from Young's inequality, which will be fixed later. By the linear growth property of $g$ and Lemma \ref{lemma of continuous random periodic solution bounded}, we know that $\norm{g(s,X_s^{-k\tau})  }_2^2$ is bounded. So we only need to estimate
\begin{align*}
&\int_{i\dt-k\tau}^{(i+1)\dt-k\tau}  \abs{e^{-A\rbrac{s+k\tau-(i+1)\dt}}  -I}^2  ds
\leq\frac{2}{3}\rbrac{\dt}^3Tr\rbrac{A^2}.
\end{align*}
By Condition (1) and Lemma \ref{lemma of order of dt between different times}, the second term in (\ref{equation of random periodic solution g part}) becomes
\begin{align*}
&\int_{i\dt-k\tau}^{(i+1)\dt-k\tau}  \norm{  g(s,X_s^{-k\tau}) - g(i\dt,X_{-k\tau+i\Delta t}^{-k\tau})}_2^2  ds\\
\leq&\int_{i\dt-k\tau}^{(i+1)\dt-k\tau} 2(C_0^2+\beta_2^2 C_4^2) \abs{s-i\dt+k\tau}ds
\leq\hk{8} \rbrac{\dt}^2.
\end{align*}
The third term follows from the global Lipschitz condition
\begin{align*}
&
\int_{i\dt-k\tau}^{\rbrac{i+1}\dt-k\tau}  \norm{  g(i\dt,X_{-k\tau+i\Delta t}^{-k\tau})  -g(i\dt,\hxit{i})   }^2_2  ds\\
\leq &
 \beta_2^2\dt  \norm{X_{-k\tau+i\Delta t}^{-k\tau}-\hxit{i}}_2^2.
\end{align*} 
Conclude the above results to obtain 
\begin{align}\label{eqn4.7}
\E\sbrac{B_2^TB_2}\leq  \hk{9}\rbrac{\dt}^2  +\frac{(1+\mu)\beta_2^2\dt}{\rbrac{\at}^2}  \norm{X_{-k\tau+i\Delta t}^{-k\tau}  -\hxit{i}}_2^2  .
\end{align}
The fifth term of (\ref{equation of expansion of sum part}) can be estimate as follows 
\begin{align*}
&\E\sbrac{  2\rbrac{X_{-k\tau+i\Delta t}^{-k\tau}-\hxit{i}}^T  \rbrac{\frac{e^{A\dt}}{\at}}  \rbrac{\frac{e^{A\dt}-I-A\dt}{\at}}   \rbrac{\hxit{i}}  }\\
\leq& 2\norm{X_{-k\tau+i\Delta t}^{-k\tau}-\hxit{i}}_2 \frac{\frac{1}{2}\abs{A^2}(\dt)^2}{(\at)^2} \norm{\hxit{i}}_2\\
\leq&\hk{10}  \rbrac{\dt}^2  \norm{X_{-k\tau+i\Delta t}^{-k\tau}-\hxit{i}}_2.
\end{align*}
To estimate the sixth term of (\ref{equation of expansion of sum part}),
\begin{align}\label{eqn4.8}
&\E\sbrac{2\rbrac{\rbrac{X_{-k\tau+i\Delta t}^{-k\tau}}^T  \rbrac{\frac{e^{A\dt}}{\at}}  -  \rbrac{\hxit{i}}^T  \rbrac{\frac{\At}{\at}}}  B_1}\\
=&\E\sbrac{2\rbrac{X_{-k\tau+i\Delta t}^{-k\tau}}^T  \rbrac{\frac{e^{A\dt}}{\at}  -\frac{\At}{\at}}    B_1}\nonumber  \\
&+\E\sbrac{2\rbrac{X_{-k\tau+i\Delta t}^{-k\tau}  -\hxit{i}}^T  \rbrac{\frac{\At}{\at}}  B_1}.\nonumber
\end{align}
Now we discuss these two terms separately,
\begin{align*}
&\E\sbrac{2\rbrac{X_{-k\tau+i\Delta t}^{-k\tau}}^T  \rbrac{\frac{e^{A\dt}}{\at}  -\frac{\At}{\at}}    B_1}
\leq  2\norm{X_{-k\tau+i\Delta t}^{-k\tau}}_2  \frac{\abs{  \frac{1}{2}  A^2\rbrac{\dt}^2}}  {\at}  \norm{B_1}_2\\
&
\leq  \hk{12}(\dt)^{7/2}  +\frac{ \sqrt{1+\mu} \beta_1\hk{11}  (\dt)^3}{(\at)^2}  \norm{  X_{-k\tau+i\Delta t}^{-k\tau}  -\hxit{i}  }_2.
\end{align*}
And,
\begin{align}
&\E\sbrac{2\rbrac{X_{-k\tau+i\Delta t}^{-k\tau}  -\hxit{i}}^T  \rbrac{\frac{\At}{\at}}  B_1}\notag\\
\leq&  \frac{2\sqrt{\hk{7}}(\dt)^{3/2}}{\at}  \norm{X_{-k\tau+i\Delta t}^{-k\tau}  -\hxit{i}}_2  \rbrac{1+\dt\abs{A}}  \label{equation of difference 1}\\
&+  \frac{2\sqrt{1+\mu}\beta_1\dt}{(\at)^2}  \norm{X_{-k\tau+i\Delta t}^{-k\tau}  -\hxit{i}}_2^2  \rbrac{1+\dt\abs{A}}.\notag
\end{align}
We use the conditional expectation to eliminate the seventh term
\begin{align*}
&\E\sbrac{ \rbrac{\rbrac{X_{-k\tau+i\Delta t}^{-k\tau}}^T  \rbrac{\frac{e^{A\dt}}{\at}}  -  \rbrac{\hxit{i}}^T  \rbrac{\frac{\At}{\at}}}  B_2   }\\
=&\E\sbrac{\rbrac{\rbrac{X_{-k\tau+i\Delta t}^{-k\tau}}^T  \rbrac{\frac{e^{A\dt}}{\at}}  -  \rbrac{\hxit{i}}^T  \rbrac{\frac{\At}{\at}}}  \E\sbrac{ B_2 |\mathcal{F}^{i\dt-k\tau}  } }\\
=&0.
\end{align*}
For the last term,
\begin{eqnarray*}
\E\sbrac{2B_1^T B_2}&\leq&2\norm{B_1^T}_2\cdot\norm{B_2}_2\\
&\leq&  \hk{13}(\dt)^{5/2}  +\hk{14}(\dt)^{3/2}  \norm{X_{-k\tau+i\Delta t}^{-k\tau}  -\hxit{i}}_2^2.
\end{eqnarray*}
Combining all the estimation above, we have 
\begin{align*}
&\frac{\abs{X_{-k\tau+(i+1)\Delta t}^{-k\tau}-\hxit{(i+1)}}^2}{\rbrac{\at}^2}-\abs{X_{-k\tau+i\Delta t}^{-k\tau}-\hxit{i}}^2\\
\leq&\rbrac{\frac{(1+\mu)\beta_2^2\dt}{\rbrac{\at}^2}  + \frac{2\sqrt{(1+\mu)}\beta_1\dt}{(\at)^2} + \hk{16}(\dt)^{3/2}}\norm{X_{-k\tau+i\Delta t}^{-k\tau}  -\hxit{i}}_2^2\\
&   +\hk{15}\rbrac{\dt}^2 + \rbrac{ \frac{2\sqrt{\hk{7}}(\dt)^{3/2}}{\at}  + \hk{17}(\dt)^2}  \norm{X_{-k\tau+i\Delta t}^{-k\tau}  -\hxit{i}}_2.
\end{align*}
Now we notice that the term $\norm{X_{-k\tau+i\Delta t}^{-k\tau}  -\hxit{i}}_2^2$ has coefficients, the largest of which contains a constant multiplied by $\dt$. The largest free term contains a constant multiplied by $(\dt)^2$. Choosing $\mu$ and $\dt$ small enough and applying Young's inequality for the term $(\dt)^{3/2}\norm{X_{-k\tau+i\Delta t}^{-k\tau}  -\hxit{i}}_2$, and from (\ref{equation of discrete rds expansion in chapter 4}) we get
\begin{eqnarray}
&&\hskip 1cm\rbrac{\at}^{-2M}  \norm{X_{-k\tau+M\Delta t}^{-k\tau}  -\hxit{M}}_2^2\label{inequality of error between exact and numerical}\\
&\leq&  \sum_{i=0}^{M-1}  \rbrac{\at}^{-2i}  \rbrac{  \hk{20}\dt    \norm{X_{-k\tau+i\Delta t}^{-k\tau}  -\hxit{i}}_2^2  +\hk{18}(\dt)^2}\notag\\
&\leq&  \hk{19}(\dt)(\at)^{-2M}  +\hk{20}(\dt)  \sum_{i=0}^{M-1}  \rbrac{\at}^{-2i}  \norm{X_{-k\tau+i\Delta t}^{-k\tau}  -\hxit{i}}_2^2,\notag
\end{eqnarray}
where
\begin{align*}
\hk{19} = \frac{\hk{18}(\at)^2  }  {2\alpha\dt-\alpha^2\rbrac{\dt}^2}(\dt)=\frac{\hk{18}(\at)^2}  {2\alpha-\alpha^2\rbrac{\dt}},\ 
\hk{20} = \frac{(1+\mu)(2\beta_1+\beta_2^2+\varepsilon)}{(\at)^2}.
\end{align*}
Here $\mu$, $\varepsilon$ and the time step $\dt$ are chosen small enough such that
$\rbrac{\hk{20}\dt+1}(\at)^2<1.$
Now using the discrete time Gronwall inequality, from (\ref{inequality of error between exact and numerical}), we have
\begin{align*}
&\norm{X_{-k\tau+M\Delta t}^{-k\tau}  -\hxit{M}}_2^2 \\
 \leq&  \hk{19}\dt  +\hk{19}  \hk{20}(\dt)^2  \frac{1-\rbrac{\rbrac{1+\hk{20}\dt}\rbrac{\at}^2}^M}{1-\rbrac{1+\hk{20}\dt}\rbrac{\at}^2} 
\leq \hk{21}\dt.
\end{align*}
We can find a constant $\hk{21}$ which is independent of $M$ and $\dt$. Finally we take $M=N+N'$, where $N\dt=k\tau$, $N'\in {\mathbb Z}$, then
\begin{eqnarray*}
\limsup_{k\rightarrow\infty}\norm{X_{N'\Delta t}^{-k\tau}-\hx_{N'\Delta t}^{-k\tau}}_2&=&  \limsup_{N\rightarrow\infty}\norm{X_{-k\tau+(N+N')\Delta t}^{-k\tau}-\hx_{-k\tau+(N+N')\Delta t}^{-k\tau}}_2 \\
&\leq& \sqrt{\hk{21}}\sqrt{\dt}.
\end{eqnarray*}
So we get the result.
\end{proof}

We have proved the estimation of error from $-k\tau$ to $N'\Delta t$ as $k\rightarrow\infty$ can be controlled under the $1/2$ order of the time-step. And the upper bound is uniform in time. The following theorem will give us a more general result, which is from $-k\tau$ to time $r$. Let $\widehat X_r^{-k\tau}$, $r>0$ be given by (\ref{eqn3.6}).

\begin{theorem}\label{theorem of error estimation from -kt to r}
Assume Conditions (A), (1) and (2). We choose $\dt=\tau/n$ for some $n\in \N$, $N=kn$. For any $r\geq 0$, there exists a constant $\widetilde{K}>0$ such that for any sufficiently small fixed $\dt$,
$$\limsup_{k\rightarrow\infty}\norm{X_r^{-k\tau}-\hx_r^{-k\tau}}_2\leq \widetilde{K}\sqrt{\Delta{t}},$$
where $X_r^{-k\tau}$ is the exact solution while $\hx_r^{-k\tau}$ is the numerical solution and $\widetilde{K}$ is independent of $\dt$ and $r$.
\end{theorem}

\begin{proof} Assume for any $r\geq 0$,
 $N'$ is the unique integer such that $N'\dt\leq r$, $(N'+1)\dt>r$.
According to the semi-flow property, we have,
\begin{align*}
X_r^{-k\tau}(\omega)-\hx_r^{-k\tau}(\omega)=&X_{r}^{N'\Delta t}(\omega)\circ X_{N'\Delta t}^{-k\tau}(\omega)-\hx_{r}^{N'\Delta t}(\omega)\circ \hx_{N'\Delta t}^{-k\tau}(\omega),
\end{align*}
where $\hx_{r}^{N'\Delta t}$ is finite time Euler approximation of solution of (\ref{equations of original period SDE}) from $N'\Delta t$ to $r$ and $\hx_{N'\Delta t}^{-k\tau}$ is defined as before. So,
\begin{eqnarray}
&&\norm{X_r^{-k\tau}-\hx_r^{-k\tau}}_2\label{inequality of differences between continuous and discrete rds}\\
&\leq& \norm{X_{r}^{N'\Delta t}\circ X_{N'\Delta t}^{-k\tau}-X_{r}^{N'\Delta t}\circ \hx_{N'\Delta t}^{-k\tau}}_2
+\norm{X_{r}^{N'\Delta t}\circ \hx_{N'\Delta t}^{-k\tau}-\hx_{r}^{N'\Delta t}\circ \hx_{N'\Delta t}^{-k\tau}}_2.\nonumber
\end{eqnarray}
For the first term on the right-hand side, by Lemma \ref{lemma of continuous and discrete random periodic solution error estimation from -kt to 0}, we have
$\norm{X_{N'\Delta t}^{-k\tau}-\hx_{N'\Delta t}^{-k\tau}}\leq K \sqrt{\dt}.$
By the continuity of $X_r^{N'\Delta t}(\cdot)$ with respect to initial values in $L^2(\Omega)$ (\cite{8}), then
\begin{align*}
\norm{X_{r}^{N'\Delta t}\circ X_{N'\Delta t}^{-k\tau}-X_{r}^{N'\Delta t}\circ \hx_{N'\Delta t}^{-k\tau}}_2
\leq C\norm{X_{N'\Delta t}^{-k\tau}-\hx_{N'\Delta t}^{-k\tau}}_2
\leq C_5\sqrt{\dt},
\end{align*}
where $C_5$ is independent of $\dt$.
For the second term on the right-hand side of (\ref{inequality of differences between continuous and discrete rds}), it is finite time Euler approximation with same initial value. By Theorem 10.2.2 in Kloeden and Platen \cite{5},
there exists a constant $C_6>0$ such that for sufficiently $\dt>0$, 
$$\norm{X_{r}^{N'\Delta t}\circ \hx_{N'\Delta t}^{-k\tau}  - \hx_{r}^{N'\Delta t}\circ \hx_{N'\Delta t}^{-k\tau}}_2  \leq  C_6\sqrt{\dt},$$
where the choice of $C_6$ is independent of $\dt$. The result follows by taking $\widetilde{K}=C_5+C_6$.
\end{proof}
\begin{corollary}
For any $r\geq 0$, the exact and numerical approximating random periodic solution of equation (\ref{equations of original period SDE}), $X_r^*$ and $\hx_r^*$, given in Theorem \ref{theorem of continuous random periodic solution convergence} and Theorem \ref{theorem of discrete random periodic solution convergence} respectively satisfy
$$\norm{X_r^*-\hx_r^*}_2 \leq \widetilde{K}\sqrt{\Delta{t}}.$$
\end{corollary}
\begin{proof}
The result follows from 
\begin{align*}
\norm{X_r^*-\hx_r^*}_2
     \leq&\limsup_{k \rightarrow \infty}\left [\norm{X_r^*-X_r^{-k\tau}}_2+  \norm{X_r^{-k\tau}-\hx_r^{-k\tau}}_2+  \norm{\hx_r^{-k\tau}-\hx_r^*}_2\right ].
\end{align*}
\hfill \end{proof}\vskip5pt

\subsection{Modified Milstein method}

For Milstein method, we can use the similar calculation as Euler-Maruyama scheme to get an improved error estimate between discrete approximate solution and the exact solution.

\begin{theorem}\label{theorem of error estimation from -kt to r for SRK1}
Assume Conditions (A), (1$^{\prime}$) and (2). 
Then there exists a constant $K^*>0$ such that for any sufficiently small fixed $\dt$, the error between the exact solution $X_r^{-k\tau}$ and the numerical solution $\hx_r^{-k\tau}$ given by Milstein scheme (\ref{equation of scheme of discrete random periodic solution by steps for RK1}) is
$\limsup_{k\rightarrow\infty}\norm{X_r^{-k\tau}-\hx_r^{-k\tau}}_2\leq K^*{\Delta{t}},$
for all $r\geq 0$, where $K^*$ is independent of $\dt$.
\end{theorem}

\begin{proof}
In the following proof, we always denote by $\hat K_\cdot$ the constant derived from the unlderlining computation unless otherwise stated.
We consider the error in the similar way as Lemma \ref{lemma of continuous and discrete random periodic solution error estimation from -kt to 0}.
\begin{eqnarray}
&&\hspace{1.5cm}\rbrac{\at}^{-2M}\abs{X_{-k\tau+M\Delta t}^{-k\tau}-\hxit{M}}^2\label{equation of discrete rds expansion in chapter 4 for SRK1}\\
&=&\sum_{i=0}^{M-1}\rbrac{\at}^{-2i}\rbrac{\frac{\abs{X_{-k\tau+(i+1)\Delta t}^{-k\tau}-\hxit{(i+1)}}^2}{\rbrac{\at}^2}-\abs{X_{-k\tau+i\Delta t}^{-k\tau}-\hxit{i}}^2}.\notag
\end{eqnarray}
For simplicity we denote
\begin{align*}
\tilde B_1=&\frac{1}{\at}\int_{i\dt-k\tau}^{(i+1)\dt-k\tau}\sbrac{e^{-A\rbrac{s+k\tau-(i+1)\dt}}f(s,X_s^{-k\tau})-f(i\dt,\hxit{i})\right.\\
&\hspace{1cm}  \left. -\int_{i\dt-k\tau}^{s} F_i^{(1)}(\hxit{i})  dW_\upsilon }ds.\\
\tilde B_2=&\frac{1}{\at}\int_{i\dt-k\tau}^{(i+1)\dt-k\tau}\sbrac{e^{-A\rbrac{s+k\tau-(i+1)\dt}}g(s,X_{s}^{-k\tau})-g(i\dt,\hxit{i})\right.\\
&\hspace{1cm}  \left. -\int_{i\dt-k\tau}^{s} G_i^{(1)}(\hxit{i})  dW_\upsilon }dW_s,
\end{align*}
with
\begin{align*}
F_i^{(1)}(x)=\frac{1}{2\sqrt{\dt}}\rbrac{f\rbrac{i\dt,\hu_+(x)}-f\rbrac{i\dt,\hu_-(x)}},\\
G_i^{(1)}(x)=\frac{1}{2\sqrt{\dt}}\rbrac{g\rbrac{i\dt,\hu_+(x)}-g\rbrac{i\dt,\hu_-(x)}}.
\end{align*}
Therefore,
\begin{eqnarray*}
&&X_{-k\tau+(i+1)\Delta t}^{-k\tau}-\hxit{(i+1)}\\
&=&e^{A\dt}X_{-k\tau+i\Delta t}^{-k\tau}-\rbrac{\At}\hxit{i}+\rbrac{\at}\rbrac{\tilde B_1+\tilde B_2}.
\end{eqnarray*}
Now we consider
\begin{eqnarray}
&&\frac{\abs{X_{-k\tau+(i+1)\Delta t}^{-k\tau}-\hxit{(i+1)}}^2}{\rbrac{\at}^2}-\abs{X_{-k\tau+i\Delta t}^{-k\tau}-\hxit{i}}^2\label{equation of expansion of sum part for SRK1}\\
&=&\rbrac{X_{-k\tau+i\Delta t}^{-k\tau}-\hxit{i}}^T  \rbrac{\frac{e^{A\dt}}{\at}-I}  \rbrac{\frac{e^{A\dt}}{\at}+I} \notag\\
&& \hskip3cm \times  \rbrac{X_{-k\tau+i\Delta t}^{-k\tau}-\hxit{i}}\notag\\
&&+\rbrac{\hxit{i}}^T  \rbrac{\frac{e^{A\dt}-I-A\dt}{\at}}^2  \rbrac{\hxit{i}}  +\tilde B_1^T\tilde B_1+\tilde B_2^T\tilde B_2\notag\\
&&+2\rbrac{X_{-k\tau+i\Delta t}^{-k\tau}-\hxit{i}}^T  \rbrac{\frac{e^{A\dt}}{\at}}  \rbrac{\frac{e^{A\dt}-I-A\dt}{\at}}   \rbrac{\hxit{i}}\notag\\
&&+2\rbrac{\rbrac{X_{-k\tau+i\Delta t}^{-k\tau}}^T  \rbrac{\frac{e^{A\dt}}{\at}}  -  \rbrac{\hxit{i}}^T  \rbrac{\frac{\At}{\at}}}  \tilde B_1\notag\\
&&+2\rbrac{\rbrac{X_{-k\tau+i\Delta t}^{-k\tau}}^T  \rbrac{\frac{e^{A\dt}}{\at}}  -  \rbrac{\hxit{i}}^T  \rbrac{\frac{\At}{\at}}}  \tilde B_2  +2\tilde B_1^T\tilde B_2\notag.
\end{eqnarray}
By the similar analysis as (\ref{eqn4.5}) and (\ref{eqn4.7}), we have
\begin{align*}
\E\sbrac{\tilde B_1^T\tilde B_1}\leq& \hk{22}\rbrac{\dt}^4  +\frac{(1+\mu)\beta_1^2\rbrac{\dt}^2}{\rbrac{\at}^2}    \norm{X_{-k\tau+i\Delta t}^{-k\tau}-\hxit{i}}_2^2 \\
&+\hk{23}^2(\dt)^3  \norm{X_{-k\tau+i\Delta t}^{-k\tau}-\hxit{i}}_2^2,
\end{align*}
and
\begin{align*}
\E\sbrac{\tilde B_2^T\tilde B_2}\leq&  \hk{24}\rbrac{\dt}^3  +\frac{(1+\mu)\beta_2^2\dt}{\rbrac{\at}^2}  \norm{X_{-k\tau+i\Delta t}^{-k\tau}  -\hxit{i}}_2^2 \\ 
&+\hk{25}(\dt)^2\norm{X_{-k\tau+i\Delta t}^{-k\tau}  -\hxit{i}}_2^2.
\end{align*}
The crossing product terms in (\ref{equation of expansion of sum part for SRK1}) are estimated similar as (\ref{eqn4.8}) as follows,
\begin{align}
&\E\sbrac{2\rbrac{\rbrac{X_{-k\tau+i\Delta t}^{-k\tau}}^T  \rbrac{\frac{e^{A\dt}}{\at}}  -  \rbrac{\hxit{i}}^T  \rbrac{\frac{\At}{\at}}}  \tilde B_1}\notag\\
\leq & \hk{26}(\dt)^{4}  +\frac{ \sqrt{1+\mu} \beta_1\hk{27}  (\dt)^3}{(\at)^2}  \norm{  X_{-k\tau+i\Delta t}^{-k\tau}  -\hxit{i}  }_2\notag\\
&+  \frac{2\sqrt{\hk{22}}(\dt)^{2}}{\at}  \norm{X_{-k\tau+i\Delta t}^{-k\tau}  -\hxit{i}}_2  \rbrac{1+\dt\abs{A}}  \label{equation of difference 2}\\
&+  \frac{2\sqrt{1+\mu}\beta_1\dt}{(\at)^2}  \norm{X_{-k\tau+i\Delta t}^{-k\tau}  -\hxit{i}}_2^2  \rbrac{1+\dt\abs{A}}\notag\\
&+2\hk{23}(\dt)^{3/2}  \norm{X_{-k\tau+i\Delta t}^{-k\tau}-\hxit{i}}_2^2 \rbrac{1+\dt\abs{A}}.\notag
\end{align}
The seventh term remain 0 under conditional expectation.
\begin{align*}
&\E\sbrac{ \rbrac{\rbrac{X_{-k\tau+i\Delta t}^{-k\tau}}^T  \rbrac{\frac{e^{A\dt}}{\at}}  -  \rbrac{\hxit{i}}^T  \rbrac{\frac{\At}{\at}}}  \tilde B_2   }=0
\end{align*}
For the last term,
\begin{eqnarray*}
\E\sbrac{2\tilde B_1^T \tilde B_2}&\leq&2\norm{\tilde B_1^T}_2\cdot\norm{\tilde B_2}_2\\
&\leq&  \hk{28}(\dt)^{7/2}  +\hk{29}(\dt)^{3/2}  \norm{X_{-k\tau+i\Delta t}^{-k\tau}  -\hxit{i}}_2^2.
\end{eqnarray*}
Combining all the estimation above, we have 
\begin{align}
&\frac{\abs{X_{-k\tau+(i+1)\Delta t}^{-k\tau}-\hxit{(i+1)}}^2}{\rbrac{\at}^2}-\abs{X_{-k\tau+i\Delta t}^{-k\tau}-\hxit{i}}^2\label{equation of combination of all terms}\\
\leq&\rbrac{\frac{(1+\mu)\beta_2^2\dt}{\rbrac{\at}^2}  + \frac{2\sqrt{(1+\mu)}\beta_1\dt}{(\at)^2} + \hk{40}(\dt)^{3/2}}\norm{X_{-k\tau+i\Delta t}^{-k\tau}  -\hxit{i}}_2^2\notag\\
&   +\hk{41}\rbrac{\dt}^3 +  \hk{42}(\dt)^2  \norm{X_{-k\tau+i\Delta t}^{-k\tau}  -\hxit{i}}_2.\notag
\end{align}
Choosing $\mu$ and $\dt$ small enough and applying Young's inequality to the term\\ $(\dt)^{2}\norm{X_{-k\tau+i\Delta t}^{-k\tau}  -\hxit{i}}_2$, and from (\ref{equation of discrete rds expansion in chapter 4 for SRK1}) we get
\begin{eqnarray}\label{eqn4.16}
&&\hskip 1cm\rbrac{\at}^{-2M}  \norm{X_{-k\tau+M\Delta t}^{-k\tau}  -\hxit{M}}_2^2\label{inequality of error between exact and numerical for SRK1}\\
&\leq&  \sum_{i=0}^{M-1}  \rbrac{\at}^{-2i}  \rbrac{  \hk{43}\dt    \norm{X_{-k\tau+i\Delta t}^{-k\tau}  -\hxit{i}}_2^2  +\hk{44}(\dt)^3}\notag\\
&\leq&  \hk{45}(\dt)^2(\at)^{-2M}  +\hk{43}(\dt)  \sum_{i=0}^{M-1}  \rbrac{\at}^{-2i}  \norm{X_{-k\tau+i\Delta t}^{-k\tau}  -\hxit{i}}_2^2,\notag
\end{eqnarray}
where
\begin{align*}
\hk{45} = \frac{\hk{44}(\at)^2  }  {2\alpha\dt-\alpha^2\rbrac{\dt}^2}(\dt)=\frac{\hk{44}(\at)^2}  {2\alpha-\alpha^2\rbrac{\dt}},\ 
\hk{43} = \frac{(1+\mu)(2\beta_1+\beta_2^2+\varepsilon)}{(\at)^2}.
\end{align*}
Here $\mu$, $\varepsilon$ and the time step $\dt$ are chosen small enough such that
$\rbrac{\hk{43}\dt+1}(\at)^2<1.$
Applying the discrete time Gronwall inequality, from (\ref{inequality of error between exact and numerical for SRK1}), we have
\begin{align}
&\norm{X_{-k\tau+M\Delta t}^{-k\tau}  -\hxit{M}}_2^2 \notag\\
 \leq&  \hk{45}(\dt)^2  +\hk{45}  \hk{43}(\dt)^2  \frac{1-\rbrac{\rbrac{1+\hk{43}\dt}\rbrac{\at}^2}^M}{1-\rbrac{1+\hk{43}\dt}\rbrac{\at}^2} 
\leq \hk{46}(\dt)^2.\label{equation of final error result}
\end{align}
We can find a constant $\hk{46}$ which is independent of $M$ and $\dt$. We take $M=N$, where $N\dt=k\tau$, then
\begin{align*}
\limsup_{k\rightarrow\infty}\norm{X_{0}^{-k\tau}-\hx_{0}^{-k\tau}}_2=  \limsup_{N\rightarrow\infty}\norm{X_{-k\tau+N\Delta t}^{-k\tau}-\hx_{-k\tau+N\Delta t}^{-k\tau}}_2  \leq \sqrt{\hk{46}}\dt.
\end{align*}
The discussion about the convergence from time $-k\tau$ to $r$ are the same as the Theorem \ref{theorem of error estimation from -kt to r} as we know that the Milstein scheme with addition term also has strong order 1.0 for finite horizon.
\vskip-0.9cm \hfill 
\end{proof}\vskip15pt

\begin{remark}
Compared with Euler-Maruyama scheme, the order 1.0 Milstein method improved the order by replacing terms $B_1$ and $B_2$ with more accurate approximation $\tilde B_1$ and $\tilde B_2$. If we did not have the additional term $$\frac{\dZ_i}{2\sqrt{\dt}}\sbrac{f\rbrac{i\dt,\huip{i})}-f\rbrac{i\dt,\huim{i}} },$$ we would only have the result with $B_1$ and $\tilde B_2$. 

Here if we compare the scheme without additional term, it is important to notice that the term $\norm{X_{-k\tau+i\Delta t}^{-k\tau}  -\hxit{i}}_2$ in (\ref{equation of difference 2}) is multiplied by $(\dt)^{2}$. But in (\ref{equation of difference 1}) it is multiplied by $(\dt)^{3/2}$. When we apply the Young's inequality in (\ref{equation of combination of all terms}), to make sure the free term with $(\dt)^3$, we have 
$$\hk{42}(\dt)^{\frac{3}{2}}  \norm{X_{-k\tau+i\Delta t}^{-k\tau}  -\hxit{i}}_2\leq \hk{47}(\dt)^3 + \hk{48}\norm{X_{-k\tau+i\Delta t}^{-k\tau}  -\hxit{i}}_2^2.$$ 
This will influence the constant $\hk{43}$ in (\ref {eqn4.16}) to fail the inequality $\rbrac{\hk{43}\dt+1}(\at)^2<1.$ On the finite horizon, $\hk{46}$ is still bounded by the boundedness of $M$. But in the case of the infinite horizon, the scheme is under the risk of instability. For this reason, we modify the scheme with the additional term from higher order scheme.
\end{remark}

\begin{corollary}
For any $r\geq 0$, the exact and numerical approximating random periodic solution of equation (\ref{equations of original period SDE}), $X_r^*$ and $\hx_r^*$, given in Theorem \ref{theorem of continuous random periodic solution convergence} and Theorem \ref{thm10} respectively satisfy
$$\norm{X_r^*-\hx_r^*}_2 \leq K^*{\Delta{t}}.$$
Here $K^*$ is independent of $\Delta t$ and $r$.
\end{corollary}

\begin{example} To illustrate the errors in Theorems \ref{theorem of error estimation from -kt to r} and 
\ref{theorem of error estimation from -kt to r for SRK1}, we simulate the random periodic solution of Example 1 with 2000 different noise realisations by both Euler-Maruyama method and modified Milstein method. We then apply Monte Carlo method to obtain the root mean square errors between the exact random periodic solution and the respective numerical schemes with 12 different step sizes: $1\times 10^{-5},2\times 10^{-5},3\times 10^{-5},4\times 10^{-5}$,$1\times 10^{-4},2\times 10^{-4},3\times 10^{-4},4\times 10^{-4}$, $1\times 10^{-3},2\times 10^{-3},3\times 10^{-3},4\times 10^{-3}$, where the exact one is given explicitly as 
$
X_t^*=\int _{-\infty}^t{\rm e}^{-(\pi+{1\over 2})(t-s)+W_t-W_s}\sin (\pi s)ds.
$
The relationship between the root mean square errors and the step size is shown in the log-log plot Fig. \ref{graph of example 2}. The difference of the orders of convergence between the Euler-Maruyama method and Milstein method is clear from the numerical simulations.
\begin{figure}
\centering
\includegraphics[scale=0.7]{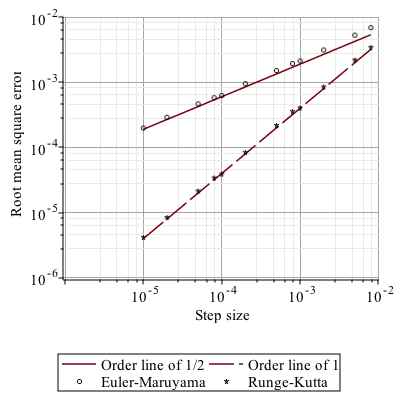}
\caption{Root mean square error versus step size as log-log plot for the SDE (\ref{equation of example 1})}\label{graph of example 2}
\end{figure}

\end{example}

\section {Periodic measures}

Let ${\cal P} (\R^m)$ denote all probability measures on $\R^m$. For $P_1, P_2\in {\cal P} (\R^m)$, define metric $d_\mathbb{L}$ as follows:
\begin{align*}
d_\mathbb{L}(P_1, P_2)
=\sup_{\varphi\in\mathbb{L}}\abs{\int_{\R^m}\varphi(x)P_1(dx) - \int_{\R^m}\varphi(x)P_2(dx)},
\end{align*}
where
$$\mathbb{L}=\{\varphi:\R^m\rightarrow\R:\abs{\varphi(x)-\varphi(y)}\leq\abs{x-y}\text{ and }\abs{\varphi(\cdot)}\leq 1\}.$$
From \cite{watanabe}, it is not difficult to prove that the metric $d_{\mathbb L}$ is equivalent to the weak topology. This useful observation was made in \cite {yuanmao}.

We can define the transition probability of the semi-flow $u$ which is generated by the solution of (\ref{zhao12}) as follows:
\begin{eqnarray}\label{5.1}
P(t+s,s,\xi,\Gamma):=P(\{\omega: u(t+s,s,\omega)\xi\in\Gamma\})=P(X_{t+s}^s(\xi)\in \Gamma),
\end{eqnarray}
for any $\Gamma\in\mathcal{B}(\R^m)$.
For any $\varphi$ being bounded and measurable 
\begin{eqnarray*}
P(t+s, s)\varphi (\xi)=\int_{\R^m} P(t+s, s, \xi, d\eta) \varphi(\eta)=E\varphi (X_{t+s}^s(\xi))
\end{eqnarray*}
defines a semigroup satisfying 
$$ P(t+s+r,s+r) \circ P(s+r, s)=P(t+s+r, s),\  r, t\geq 0,\ s\in \R.$$

Recall the following definition of periodic measure given in \cite{feng-zhao} .
\begin{definition} (\cite{feng-zhao})
The measure function $\rho_{\cdot}: \R\to {\cal P} (\R^m)$ is called periodic measure if it satisfies for any $s\in \R,\ t\geq 0$, and $\Gamma\in {\cal B}(\R^m)$,
\begin{eqnarray*}
\rho_{s+\tau}=\rho_s,\
\int_{\R^m} P(t+s, s, x, \Gamma)\rho_s(dx)=\rho_{t+s}(\Gamma).
\end{eqnarray*}
\end{definition}

From Theorem \ref{theorem of continuous random periodic solution convergence}, we know that the random periodic solution of (\ref{zhao12}) exists. So by the result in \cite{feng-zhao}, we know that the periodic measure $\rho_.$ exists, which can be defined as  the law of random periodic solutions, i.e.
\begin{eqnarray}\label{5.2}
 \rho_r(\Gamma)=P(X_r^*\in\Gamma).
 \end{eqnarray}

Similarly, we can define the transition probability of the discrete semi-flow $\hat u$ from Euler-Maruyama scheme by 
\begin{eqnarray}\label{5.3}
\hat P(t+s,s,\xi,\Gamma):=P(\{\omega: \hat u(t+s,s,\omega)\xi\in\Gamma\})=P(\hat X_{t+s}^s(\xi)\in \Gamma).
\end{eqnarray}
For any $\varphi$ being bounded and measurable 
\begin{eqnarray*}
\hat P(t+s, s)\varphi (\xi)=\int_{\R^m} \hat P(t+s, s, \xi, d\eta) \varphi(\eta)=E\varphi (\hat X_{t+s}^s(\xi))
\end{eqnarray*}
defines a semigroup satisfying 
$$ \hat P(t+s+r,s+r) \circ \hat P(s+r, s)=\hat P(t+s+r, s),\  r, t\geq 0,\ s\in \R,$$
Similar to the result in \cite {feng-zhao}, the measure function defined by 
\begin{eqnarray}\label{5.4}
 \hat\rho_r(\Gamma)=P(\hat X_r^*\in\Gamma),
 \end{eqnarray}
 is a periodic measure for Markov semigroup $\hat P(t+s,s)$. It satisfies for any $s\in \R,\ t\geq 0$, and $\Gamma\in {\cal B}(\R^m)$,
  \begin{eqnarray*}
\hat \rho_{s+\tau}=\hat \rho_s,\
\int_{\R^m} \hat P(t+s, s, x, \Gamma)\hat \rho_s(dx)=\hat \rho_{t+s}(\Gamma).
\end{eqnarray*}

 We have following error estimate of $\rho_.$ and $\hat \rho_.$. Consider the Euler-Maruyama scheme (\ref{{equation of scheme of discrete random periodic solution by steps}}) first.
\begin{theorem}\label{thm16}
Assume Conditions (A), (1) and (2). Then periodic measures $\rho_.$ and $\hat \rho_.$ of the Markov semigroup generated by the exact solution of (\ref{equations of original period SDE}) and the approximation (\ref{{equation of scheme of discrete random periodic solution by steps}}) are weak limits of transition probabilities along integral multiples of period, i.e.
\begin{eqnarray}\label{eqn5.5a}
P(r,-k\tau,\xi)\to \rho_r,\ \hat P(r,-k\tau,\xi)\to \hat \rho_r,\ {\rm as\ k\to \infty},
\end{eqnarray}
weakly and the error estimate is
\begin{eqnarray}\label{eqn5.6aa}
d_\mathbb{L}(\rho_r,\widehat{\rho}_r)\leq \widetilde{K}\sqrt{\dt},
\end{eqnarray}
where $\widetilde{K}$ is independent of $\dt$ and $r$.
\end{theorem}
\begin{proof}
To prove  (\ref{eqn5.5a}), by (\ref {5.1}), (\ref{5.2}), Theorem \ref{theorem of continuous random periodic solution convergence} and Jensen's inequality, we have 
\begin{eqnarray*}
&& d_{\mathbb L}(P(r,-k\tau,\xi), \rho_r)\\
&=& \sup_{\varphi\in\mathbb{L}}\abs{\int_{\R^m}\varphi(x)P(r,-k\tau,\xi,dx)-\int_{\R^m}\varphi(x)\rho_r(dx)}\\
&=&\sup_{\varphi\in\mathbb{L}}\abs{\E[\varphi(X_r^{-k\tau})-\varphi(X_r^*)]}\\
&\leq & \sup_{\varphi\in\mathbb{L}}\E|\varphi(X_r^{-k\tau})-\varphi(X_r^*)|\\
&\leq&\E\abs{X_r^{-k\tau}-X_r^*}\\
&\leq &\norm{X_r^{-k\tau}-X_r^*}_2\\
&\to &0,
\end{eqnarray*}
as $k\to \infty$. So $P(r,-k\tau,\xi)\to \rho_r$ weakly as $k\to\infty$ from the well known result in \cite{watanabe}. Similarly, we can have for the discrete system, 
$\hat P(r,-k\tau,\xi)\to \hat \rho_r$ weakly as $k\to \infty$.
\
Now we consider the metric between these two periodic measures $\rho_.$ and $\hat \rho_.$,
\begin{align}\label{feng62}
&d_\mathbb{L}(\rho_r,\widehat{\rho}_r)\nonumber\\
=&\sup_{\varphi\in\mathbb{L}}\abs{\int_{\R^m}\varphi(x)\rho_r(dx) - \int_{\R^m}\varphi(x)\widehat{\rho}_r(dx)}\nonumber\\
\leq&\sup_{\varphi\in\mathbb{L}}\abs{\int_{\R^m}\varphi(x)\rho_r(dx)
- \int_{\R^m}\varphi(x)P(r,-k\tau,\xi,dx)}\nonumber\\
&+\sup_{\varphi\in\mathbb{L}}\abs{\int_{\R^m}\varphi(x) P(r,-k\tau,\xi,dx) 
- \int_{\R^m}\varphi(x)\widehat{P}(r,-k\tau,\xi,dx)}\\
&+\sup_{\varphi\in\mathbb{L}}\abs{\int_{\R^m}\varphi(x) \widehat{P}(r,-k\tau,\xi,dx) 
- \int_{\R^m}\varphi(x)\widehat{\rho}_r(dx)}\nonumber\\
=&\sup_{\varphi\in\mathbb{L}}\abs{\E[\varphi(X_r^*)-\varphi(X_r^{-k\tau})]}
+\sup_{\varphi\in\mathbb{L}}\abs{\E[\varphi(X_r^{-k\tau})-\varphi(\hx_r^{-k\tau})]}\nonumber\\
&+\sup_{\varphi\in\mathbb{L}}\abs{\E[\varphi(\hx_r^{-k\tau})-\varphi(\hx_r^{*})]}\nonumber\\
\leq&\E\abs{X_r^*-X_r^{-k\tau}} + \E\abs{X_r^{-k\tau}-\hx_r^{-k\tau}} + \E\abs{\hx_r^{-k\tau}-\hx_r^{*}}\nonumber\\
\leq& \norm{X_r^*-X_r^{-k\tau}}_2+\norm{X_r^{-k\tau}-\hx_r^{-k\tau}}_2+\norm{\hat X_r^{-k\tau}-X_r^*}_2. \nonumber
\end{align}
By Theorems \ref{theorem of continuous random periodic solution convergence}, \ref {theorem of discrete random periodic solution convergence}, \ref{theorem of error estimation from -kt to r}, we have for any $\epsilon>0$, there exists $N>0$ such that when $k\geq N$, 
$$ \norm{X_r^*-X_r^{-k\tau}}_2\leq{\epsilon\over 3}, \ \norm{\hat X_r^{-k\tau}-X_r^*}_2\leq {\epsilon\over 3},$$
and 
$$ \norm{X_r^{-k\tau}-\hx_r^{-k\tau}}_2\leq\widetilde{K}\sqrt{\dt}+{\epsilon\over 3}.$$
Then taking $k\geq N$ in (\ref{feng62}), we have
\begin{align*}
d_\mathbb{L}(\rho_r,\widehat{\rho}_r)\leq\widetilde{K}\sqrt{\dt}+\epsilon.
\end{align*}
Note in the above inequality, the left hand side does not depend on $k$ and $\epsilon$ is arbitrary. So (\ref{eqn5.6aa}) is obtained.
\end{proof}
\begin{remark}
There are a number of work about approximating of invariant measures for SDE using Euler-Maruyama method and Milstein method (\cite{mattingly}, \cite{talay}, \cite{talay-tubaro}, \cite{yuanmao}). For finite horizon, the order of weak convergence of Euler-Maruyama method was proved to be 1.0, a significant improvement from the order $0.5$ in the strong convergence (c.f. \cite{5}). However, the order of 1.0 is not guaranteed in the infinite horizon case, see \cite{mattingly} for the case of the invariant measures. On the other hand, in some work such as \cite{talay}, \cite{talay-tubaro}, the order of error of Euler-Maruyama method was managed to increase to 1.0 under the non-degenerate condition. Here we do not have such an assumption, and we have order 0.5 in the weak convergence formulation. However, in the case of the modified 
Milstein method,  we will see that the error is  of order 1.0 in the next theorem. 
Note that the error estimate with the Milstein scheme is also 1.0 in the weak convergence formulation even in the 
non-degenerate case (\cite{talay}, \cite{talay-tubaro}).
\end{remark}
\begin{theorem}
Assume Condition (A), ($1^\prime$) and (2). Consider the modified Milstein scheme (\ref{equation of scheme of discrete random periodic solution by steps for RK1}). Then  the periodic measure $\hat \rho_.$ of the Markov semi-groups generated by the discretised
semi-flow is the weak limit of its transition probability along integral multiples of period, i.e.
$$\hat P(r,-k\tau,\xi)\to \hat \rho_r,\ {\rm as\ k\to \infty},$$
weakly and the error estimate between the approximating periodic measure $\hat \rho_.$ and the exact periodic measure is
 $$d_\mathbb{L}(\rho_r,\widehat{\rho}_r)\leq K^*{\dt},$$
where $ {K^*}$ is independent of $\dt$ and $r$.
\end{theorem}
\begin{proof}
The proof is similar to the proof of Theorem \ref{thm16}, but using Theorem \ref{theorem of error estimation from -kt to r for SRK1} instead of Theorem \ref{theorem of error estimation from -kt to r}.
\end{proof}

\section{Transformation of the periodic SDE via Lyapunov-Floquet transformation}
In this section, we consider the following m-dimensional system 
\begin{eqnarray}
dX_t^{t_0}&=A(t)X_t^{t_0}dt  + \widetilde{f}(t,X_t^{t_0})dt +\widetilde{g}(t,X_t^{t_0})dW_t,\ t\geq t_0,
\label{equations of original period SDE with At} 
\end{eqnarray}
with $X_{t_0}^{t_0}=\xi$.
We assume that the matrix $A(t)$ is a continuous $\tau$-periodic $m\times m$ real matrix and the functions $\widetilde{f}$ and $\widetilde{g}$ are both $\tau$-periodic in time, i.e.
\begin{align*}
A(t+\tau)=A(t),\ \widetilde{f}(t+\tau,\cdot)=\widetilde{f}(t,\cdot),\ \widetilde{g}(t+\tau,\cdot)=\widetilde{g}(t,\cdot),\ \text{for any}\ t\in\R.
\end{align*}

To solve this problem we need to apply the Floquet theorem to transfer this system to a system with the linear part having a time invariant generator.

\subsection{The transformation}
The well known Floquet theorem can be found in many books, such as \cite{15}. It says that 
if$\ $ $\Phi(t)$ is a fundamental matrix solution of the periodic system $\dot{X}=A(t)X$, then so is $\Phi(t+\tau)$. Moreover, there exists an invertible $\tau$-periodic matrix $P(t)$ such that
$\Phi(t)=P(t)e^{Rt},$
where $R$ is a constant matrix.
The matrix $P(t)$ is called the Lyapunov-Floquet transformation matrix and $X=P(t)Z$ is called the Lyapunov-Floquet transformation.
\begin{proposition}\label{corollary of floquet theorem}
Under Lyapunov-Floquet transformation $X(t)=P(t)Z(t)$, the periodic system (\ref{equations of original period SDE with At}) is transferred to the following system with 
constant coefficient matrix linear part
\begin{eqnarray}
dZ_t^{t_0} &= RZ_t^{t_0}dt + P(t)^{-1} \widetilde{f}(t,P(t)Z_t^{t_0})dt + P(t)^{-1} \widetilde{g}(t,P(t)Z_t^{t_0})dW_t,
\label{equation of transferred from At into constant matrix R 1} 
\end{eqnarray}
with $Z_{t_0}^{t_0}=P(t_0)^{-1}\xi$.
\end{proposition}
\begin{proof}
The proof follows some elementary calculations.
\end{proof}
%

From the periodicity of $P$, we know that$$\Phi(t+\tau)=P(t+\tau)e^{R(t+\tau)}=P(t)e^{Rt}e^{R\tau}=\Phi(t)e^{R\tau}.$$
Since $e^{R+2\pi kiI}=e^R e^{2\pi ki }I=e^R$ for any $k\in \Z$, the constant matrix $R$ is not unique. It is also not necessarily real, even if $e^{R\tau}$ is real. So we need the following corollary to guarantee such a real constant matrix exists.
\begin{corollary}
Let
$B=\frac{R+\overline{R}}{2},\ S(t)=\Phi(t)e^{-Bt}.$
Then $S(t)$ is real and $2\tau$-periodic. Under the transformation $X_t^{t_0}=S(t)Z_t^{t_0}$, the periodic system (\ref{equations of original period SDE with At}) is transferred to the following system with 
constant coefficient matrix linear part
\begin{eqnarray}
dZ_t^{t_0} &= BZ_t^{t_0}dt + S(t)^{-1} \widetilde{f}(t,S(t)Z_t^{t_0})dt + S(t)^{-1} \widetilde{g}(t,S(t)Z_t^{t_0})dW_t,
\label{equation of transferred from At into constant matrix R} 
\end{eqnarray}
with $Z_{t_0}^{t_0}=S(t_0)^{-1}\xi$,
\end{corollary}
\begin{proof}
Because $A(t)$ is real, so the matrix $C=e^{R\tau}=\Phi(\tau)\Phi^{-1}(0)$ is real. Thus for the real matrix $B=\frac{R+\overline{R}}{2}$, 
$C^2=e^{R\tau}e^{\overline{R}\tau}=e^{2B\tau}.$
Note $S(t)$ is real since $B$ is real. And notice that
\begin{align*}
S(t+2\tau)=\Phi(t+2\tau)e^{-B(t+2\tau)}=\Phi(t)C^2 e^{-2B\tau} e^{-Bt}=\Phi(t) e^{-Bt}=S(t).
\end{align*}
Then we can obtain the time invariant system in a similar way as in the Corollary \ref{corollary of floquet theorem}. The only difference is that the system with real constant coefficient matrix linear part becomes $2\tau$-periodic.
\end{proof}

\subsection{Convergence theorem of the periodic parameter matrix system}
\hskip -0.6cm {\bf Condition (A$'$)}. {\it The matrix function $A(t)$ is $\tau$-periodic, the corresponding matrix $B$ is symmetric with eigenvalues satisfying 
$0>\lambda_1\geq\lambda_2\geq\ldots\geq\lambda_m.$}

Because $S(t)$ is continuous and periodic, so $S(t)$ is bounded. The periodicity and continuity of $S(t)^{-1}$ is obtained from the properties of $S(t)$, it is concluded that $S(t)^{-1}$ is bounded as well. Thus there exists a constant $M$ such that
$\abs{S(t)^{-1}}\abs{S(t)}\leq \gamma.$
For the periodic parameter matrix system, we give the following condition \\
\textbf{Condition (1$'$)}. {\it Assume there exists a constant $\tau>0$ such that for any $t\in \R$, $x\in \R^m$, 
$\widetilde f(t+\tau,x)=\widetilde f(t,x),\ 
\widetilde g(t+\tau,x)=\widetilde g(t,x).$
There exist constant $\widetilde {C_0}, \widetilde{\beta_1},\widetilde{\beta_2}>0$ with $\widetilde{\beta_1}\gamma+\frac{\widetilde{\beta_2}^2\gamma^2}{2}<\abs{\lambda_1}$, such that for any $s,t\in\R$ and $x,y\in\R^m$, 
\begin{align*}
\abs{\widetilde{f}(s,x)-\widetilde{f}(t,y)}\leq\widetilde{C_0}\abs{s-t}^{1/2}+\widetilde{\beta_1}\abs{x-y},&\\
\abs{\widetilde{g}(s,x)-\widetilde{g}(t,y)}\leq \widetilde{C_0}\abs{s-t}^{1/2}+\widetilde{\beta_2}\abs{x-y}.
\end{align*}
From this condition it follows that for any $x\in\R^m$, the linear growth condition also holds
$\abs{\widetilde{f}(t,x)}\leq\widetilde{\beta_1}\abs{x}+\widetilde{C_1},\ \abs{\widetilde{g}(t,x)}\leq\widetilde{\beta_2}\abs{x}+\widetilde{C_2},$
where the constants $\widetilde{C_1},\widetilde{C_2}>0$, which are independent of time $t$.
}

\begin{theorem}\label{theorem of continuous random periodic solution convergence with At}
Assume that Conditions ($A^{\prime}$), ($1^{\prime}$). Then there exists a unique random periodic solution $X_r^*\in L^2(\Omega)$ of period $2\tau$ such that for any initial value $\xi(\omega)$ satisfying Condition (2), the solution of (\ref{equations of original period SDE with At}) satisfies
$\lim\limits_{k\rightarrow\infty}\norm{X_r^{-2k\tau}(\xi)-X_r^*}_2=0.$
\end{theorem}
\begin{proof}
We only need to verify that the corresponding time invariant system
\begin{eqnarray}\label{eqn5.5}
dZ_t^{t_0} &= BZ_t^{t_0}dt + f\rbrac{t,Z_t^{t_0}}dt + g\rbrac{t,Z_t^{t_0}}dW_t,
\end{eqnarray}
with $Z_{t_0}^{t_0}=S(t_0)^{-1}\xi$,
where
$f\rbrac{t,x}=S(t)^{-1}\widetilde{f}(t,S(t)x),\ g\rbrac{t,x}=S(t)^{-1}\widetilde{g}(t,S(t)x),$
satisfies the conditions of Theorem \ref{theorem of continuous random periodic solution convergence}. It is easy to see that
$f(t+2\tau,x)=f(t,x),\ 
g(t+2\tau,x)=g(t,x).
$
For Condition (1), the largest eigenvalue of the matrix $B$ is $\lambda_1$. By the Lipschitz condition on function $\widetilde{f}$ and $\widetilde{g}$, we have following result in the time invariant system
$\abs{f(t,x)-f(t,y)}
\leq \widetilde{\beta_1}\gamma\abs{x-y}.$
This means the function $f$ will preserve the Lipschitz property with constant $\beta_1=\widetilde{\beta_1} \gamma$. Similarly we can prove that the function $g$ possesses the Lipschitz condition with constant $\beta_2=\widetilde{\beta_2} \gamma$. Meanwhile, from Condition (1$^{\prime}$), we have 
$\beta_1+\frac{\beta_2^2}{2}<\abs{\lambda_1}.$
Moreover, for any $x\in\R^m$,
\begin{align*}
\abs{f(t,x)}=\abs{S(t)^{-1}\widetilde{f}(t,S(t)x)}\leq\widetilde{\beta_1}\abs{S(t)^{-1}}\abs{S(t)x}+\abs{S(t)^{-1}}\widetilde{C_1 }\leq \beta_1\abs{x} + C_1.
\end{align*}
Therefore we can verify the linear growth property of $f$ and $g$ with the constants $C_1,C_2>0$. The constants $\beta_1$ and $\beta_2$ are both independent of time $t$.
For Condition (2), the initial value of the time invariant system will preserve the boundedness because of the boundedness of $S(t)^{-1}$.
According to Theorem \ref{theorem of continuous random periodic solution convergence}, there exists a random periodic solution $Z_r^*\in L^2(\Omega)$ with period $2\tau$ such that
$\lim_{k\rightarrow\infty}\norm{Z_r^{-2k\tau}(\xi)-Z_r^*}_2=0.$
It turns out that
\begin{align*}
\lim_{k\rightarrow\infty}\norm{X_r^{-2k\tau}(\xi)-X_r^*}_2 
\leq \norm{S(r)}\lim_{k\rightarrow\infty}\norm{Z_r^{-2k\tau}(\xi)-Z_r^*}_2
=0.
\end{align*}
The $2\tau$-periodicity of $S(r)$ and $Z_r^{-2k\tau}$ give us the random periodicity of solution $X^*(r,\omega)$. So $X_r^*$ is a random periodic solution of (\ref{equations of original period SDE with At}) of period $2\tau$.
\end{proof}

\subsection{Numerical approximation scheme and error estimate}
With the existence of the random periodic solutions, we now consider the scheme to simulate the process $Z$ of equation (\ref{equation of transferred from At into constant matrix R}). Similar as before, we can consider strong and weak convergence in Euler-Maruyama and modified Milstein methods. Due to the length of
the paper,  we only consider strong convergence in the Euler scheme given by
\begin{eqnarray}&&\widehat Z^{-2k\tau}_{-2k\tau+(i+1)\Delta t}\notag\\
&=&\widehat Z^{-2k\tau}_{-2k\tau+i\Delta t}+[B\widehat Z^{-2k\tau}_{-2k\tau+i\Delta t}+S(i\dt)^{-1}\widetilde{f}(i\dt,S(i\dt)\widehat Z^{-2k\tau}_{-2k\tau+i\Delta t})]\dt\notag\\
                   &&+S(i\dt)^{-1}\widetilde{g}(i\dt,S(i\dt)\widehat Z^{-2k\tau}_{-2k\tau+i\Delta t})\rbrac{W_{-2k\tau+(i+1)\dt}-W_{-2k\tau+i\dt}}\label{equation of discrete random periodic solution with At}.
\end{eqnarray}

\begin{theorem}\label{thm5.3}
Assume Conditions ($A^{\prime}$), ($1^{\prime}$) and (2), $S(t)\in C^1(\R)$. Then there exists $\widehat{Z}^*_r$, which is a random periodic solution of period $2\tau$ for discrete random dynamical system generated from (\ref{eqn5.5}), such that
$$\lim_{k\rightarrow\infty}\norm{X_r^{-2k\tau}-S(r)\widehat{Z}_{r}^{-2k\tau}}_2\leq \widetilde{K}\sqrt{\Delta { t}},\ 
{\rm and}\ 
\norm{X_r^*-S(r)\widehat{Z}_{r}^*}_2\leq \widetilde{K}\sqrt{\Delta { t}},\ r\in[0,T],
$$
for a constant $\widetilde{K}>0$, which is independent of $\dt$, where $X_r^*$ is the exact random periodic solution of (\ref{equations of original period SDE with At}).
\end{theorem}

\begin{proof}
By Theorem \ref {theorem of discrete random periodic solution convergence}, there exists $\hat Z_r^*\in L^2(\Omega)$ such that 
$\limsup\limits_{k \rightarrow \infty}\norm{\widehat{Z}_{r}^{-2k\tau}-\widehat{Z}_{r}^*}_2\\=0,$
where $\widehat{Z}^*_r$ is the random periodic solution of period $2\tau$ for discrete random dynamical system generated from (\ref{eqn5.5}).
According to Theorem \ref{theorem of error estimation from -kt to r}, we have the conclusion that there exists a constant $K_1>0$ such that
$\lim_{k\rightarrow\infty}\norm{X_r^{-2k\tau}-S(r)\widehat{Z}_{r}^{-2k\tau}}_2\leq K_1\norm{S(r)}_2\sqrt{\Delta {\tilde t}}\leq \widetilde K\sqrt{\Delta {t}}.$
Thus it follows
that 
\begin{align*}
\norm{X_r^*-S(r)\widehat{Z}_{r}^*}_2
     \leq&\limsup_{k \rightarrow \infty}\norm{X_r^*-X_r^{-2k\tau}}_2+  \limsup_{k \rightarrow \infty}\norm{X_r^{-2k\tau}-S(r)\widehat{Z}_{r}^{-2k\tau}}_2\\
     &+  \limsup_{k \rightarrow \infty}\norm{S(r)\widehat{Z}_{r}^{-2k\tau}-S(r)\widehat{Z}_{r}^*}_2
     \leq \widetilde{K}\sqrt{\Delta{t}}. 
\end{align*}
\vskip-0.9cm \hfill \end{proof}\vskip15pt

\section*{Acknowledgement} We sincerely thank the anonymous referees for their constructive comments, which result in significant improvements of the paper. CF and HZ would like to acknowledge the financial support of Royal Society 
Newton Advanced Fellowship grant NA150344.

\bibliographystyle{siamplain}
\bibliography{references}

\end{document}